\documentclass[a4paper]{article}

\usepackage[all]{xy}\usepackage[latin1]{inputenc}        
\usepackage[dvips]{graphics,graphicx}
\usepackage{amsfonts,amssymb,amsmath,color,mathrsfs, amstext,bm}
\usepackage{amsbsy, amsopn, amscd, amsxtra, amsthm,authblk, enumerate}
\usepackage{upref}
\usepackage[colorlinks,
            linkcolor=red,
            anchorcolor=red,
            citecolor=red
            ]{hyperref}

\usepackage{geometry}
\usepackage[pagewise]{lineno}

\usepackage{float}

\usepackage{yhmath}

\usepackage[normalem]{ulem}

\usepackage[ruled,linesnumbered]{algorithm2e}

\usepackage{authblk}

\usepackage{amscd}

\usepackage{graphicx}

\usepackage{subfigure}

\numberwithin{equation}{section}

\def\E{\mathbb{E}}

\newtheorem{theorem}{Theorem}
\newtheorem{lemma}{Lemma}

\newtheorem{proposition}{Proposition}
\newtheorem{remark}{Remark}

\newtheorem{definition}{Definition}

\begin{document}
\title{On the mean-field limit of the Cucker-Smale model with Random Batch Method}
        \author[a]{Yuelin Wang \thanks{sjtu$\_$wyl@sjtu.edu.cn}}
        \author[a]{Yiwen Lin \thanks{linyiwen@sjtu.edu.cn}}
\affil[a]{School of Mathematical Sciences, MOE-LSC, Shanghai Jiao Tong University, Shanghai, 200240, P.R.China }
\date{}
\maketitle

\begin{abstract}
    In this work, we focus on the mean-field limit of the Random Batch Method (RBM) for the Cucker-Smale model. Different from the classical mean-field limit analysis, the chaos in this model is imposed at discrete time and is propagated to discrete time flux. We approach separately the limits of the number of particles $N\to\infty$ and the discrete time interval $\tau\to 0$ with respect to the RBM, by using the flocking property of the Cucker-Smale model and the observation in combinatorics. The Wasserstein distance is used to quantify the difference between the approximation limit and the original mean-field limit. Also, we combine the RBM with generalized Polynomial Chaos (gPC) expansion and proposed the RBM-gPC method to approximate stochastic mean-field equations, which conserves positivity and momentum of the mean-field limit with random inputs.\\
    

\end{abstract}

\section{Introduction}

    Collective behaviors of many-body systems are ubiquitous in the real world, like  flocking of birds \cite{ha2009simple,Ha2008particle,cucker2007emergent,motsch2011new}, swarming of fishes \cite{toner1998flocks}, synchronicity of fireflies \cite{kuramoto1975self,buck1966biology} and pacemaker cells \cite{peskin1975mathematical}. We use the jargon ``flocking" to describe the phenomenon in which self-propelled particles organize into an ordered motion using only limited environmental information and simple rules \cite{toner1998flocks}. Since there is a great deal of literature on collective behaviors and related models, we refer the readers to \cite{degond2008continuum,justh2002simple,albi2019vehicular,bellomo2017quest,bellomo2012mathematical,vicsek1995novel,ha2016collective,dorfler2014synchronization,topaz2004swarming} and the references therein.

    The Cucker-Smale model, proposed by Cucker and Smale \cite{cucker2007emergent}, is a famous model of collective behaviors that models phenomenologically  the flocking phenomenon. It is in the form of an $N$-body second-order system of ordinary differential equations or particle systems  for position and velocity, which resemble Newton's equations.  

     In addition, the mean-field limit, as the number of particles $N \to \infty$,  is also a prevalent tool for simplifying large particle systems, which makes use of continuum models derived within the framework of mathematical kinetic theory, to approximate collective behaviors. In these years, there has been extensive research on the mean-field limit of particle systems and we refer to \cite{jabin2014review,chaintron2022propagation,chaintron2022propagation2} for survey articles.
     
    As a classical type of $N$-body systems, in the Cucker-Smale model, each particle interacts with $N-1$ particles, and thus the computational cost to solve it is $\mathcal{O}(N^2)$ per time step. An efficient algorithm to reduce computational complexity to $\mathcal{O}(N)$ is the Random Batch Method (RBM) proposed by Jin et al. \cite{jin2020random} in 2020. The RBM indeed constructs a randomly decoupled system with subsystems of interaction between $p$ particles, where the constant $p\ll N$. In this system, at each time step, the interaction only occurs within small batches with $p$ particles. The choice of batches is random and thus the time-average effect makes it  a good approximation of the original system \cite{jin2020random,ha2021uniform,jin2022random}.

    An interesting problem is to understand the mean-field limit of the RBM. In 2022, Jin and Li considered the mean-field limit of the RBM on first-order systems with Gaussian noise in \cite{jin2022mean}. 
    Inspired by \cite{jin2022mean}, in this paper, we investigate the mean-field limit of the RBM on the Cucker-Smale model as the number of particles tends to infinity. Unlike the first-order system considered in \cite{jin2022mean}, the Cucker-Smale model is a second-order system with interaction coupled with both velocity and position. Technically, moment control needs to be  derived by the asymptotic flock estimate instead of using the contraction property in first-order systems in \cite{jin2020random,jin2022random}. Specifically, we mainly analyze the path from (b) to (c), and briefly discuss the path from (a) to (d) in the following figure \eqref{graph}. On the upper left corner, Ha et al. give an estimate on (a) in \cite{ha2021uniform} and a proof of the existence of mean-field limit on (d) in \cite{ha2018uniform}. Their analysis on (d) is uniform in time but implicit in $N.$ In later work \cite{natalini2020mean}, an explicit estimate on $N$ is given but only works on finite time. The four models in \eqref{graph} are represented in Section \ref{sec:prelimiaries}.
    
    \begin{equation}\label{graph}
        \begin{CD}
        \text{Cucker-Smale system} \,\eqref{eq: cs} @ > (d)\, N\to\infty >> \text{mean-field limit} \,\eqref{eq: mfl for cs}\\
         @A (a)\, \tau\to 0 AA  @AA (c)\, \tau\to 0 A\\
        \text{Random Batch Cucker-Smale model} \,\eqref{eq: the RBMcs} @>> (b)\, N\to\infty > \text{mean-field limit of RBM} \,\eqref{eq: the RBM limit discrete}
        \end{CD}
    \end{equation}

   With the justification of (b) and (c),   we employ the RBM, which is an efficient particle-based numerical scheme for the approximation of stochastic mean-field equations of the Cucker-Smale model, inspired by \cite{carrillo2019particle}. We consider such an application in the stochastic setting, where the interacting force depends on a random variable modeling uncertainties.  This is an important problem in Uncertainty Quantification \cite{jin2017uncertainty}. The method combines the RBM with a generalized Polynomial Chaos (gPC) expansion in the random space and thus we call it the RBM-gPC. It avoids loss of positivity like MCgPC in \cite{carrillo2019particle}, and moreover preserves the mean velocity during evolution.

    The rest of the paper is organized as follows. In Section \ref{sec:prelimiaries}, we give a concise introduction to the Cucker-Smale model including notations, assumptions and prerequisites. Section \ref{sec:main} presents our main results on the analysis of limits (b) and (c) in \eqref{graph} with helpful discussions on the limits from (a) to (d). Section \ref{sec: mainpf1} and Section \ref{sec: mainpf2} provide the proof details. In Section \ref{sec:the RBMgPC}, we show some numerical experiments with the RBM-gPC. Finally, Section \ref{sec:conclusion} is devoted to a summary of our main results and some remaining issues to be explored in the future.
    
    

\section{Preliminaries}\label{sec:prelimiaries}
In this section, we introduce the main models in \eqref{graph} with assumptions and notations. Also, some basic properties of the models are given.

\subsection{The Cucker-Smale model}\label{subsec: CSmodel}
Let $X_i$ and $V_i$ be the position and velocity of the $i$-th particle with unit mass, and $\psi(| X_j -X_i|)$ be the communication weight between the $j$-th and $i$-th particles. The Cucker-Smale model reads as the following
\begin{equation} \label{eq: cs}
    \left\{\begin{aligned}
        \frac{d}{dt} X_i(t) =& V_i(t),\quad i=1,\cdots, N,\\
        \frac{d}{dt} V_i(t) =& \frac{\kappa}{N-1}\sum\limits_{j=1}^N 
        \psi(|X_j(t)-X_i(t)|)(V_j(t)-V_i(t)),
    \end{aligned}\right.
\end{equation}
where $\kappa$ is the nonnegative coupling strength and $\psi$ satisfies  positivity, boundedness, Lipschitz continuity and mononticity conditions, i.e., there exist positive constants $\psi_0, \psi_M  > 0 $  such that 
\begin{equation}\label{ass:psi}
    0< \psi_0 \le \psi(r) \le \psi_M,\, \forall r\ge0; \quad \|\psi\|_\text{{Lip}}<\infty;\quad (\psi(r_1)-\psi(r_2))(r_1-r_2)\le 0,\, r_1,r_2 \in \mathbb{R}_+.
\end{equation}
Without loss of generality, we set $\psi_M = 1$ in this paper for convenience. 
Under the condition \eqref{ass:psi} of $\psi,$ it is observed in \cite{cucker2007emergent,ha2009simple} that the total momentum is conserved as a constant and the total energy is nonincreasing along the Cucker-Smale flow. Actually,
    supposing that $\{(X_i,V_i)\}$ is the solution of  system \eqref{eq: cs}, then for any $t>0,$ one has
    \begin{equation}\label{eq:CSproperty}
        \frac{d}{dt}\sum\limits_{i=1}^N V_i (t) = 0, \quad \frac{d}{dt}\sum\limits_{i=1}^N |V_i(t)|^2 = -\frac{\kappa}{N-1}\sum\limits_{i,j}\psi(|X_j(t) - X_i(t)|) |V_j (t) - V_i (t) |^2.
    \end{equation}

A key observation of the Cucker-Smale model, given by Ha et al. in \cite{ha2009simple}, is that the standard deviations of particle phase-space positions are dominated by SDDI (the system of dissipative differential inequalities):
\begin{equation}\label{origin SDDI}
    \left| \frac{d\mathcal{X}}{dt} \right| \le \mathcal{V}, \quad \frac{d\mathcal{V}}{dt} \le -\phi (\mathcal{X}) \mathcal{V},
\end{equation}
where $(\mathcal{X},\mathcal{V})$ are nonnegative functions and $\phi$ is a nonnegtive measurable function. This is a useful tool to analyze the Cucker-Smale model. In this paper, we use a variant  of \eqref{origin SDDI} (Lemma \ref{lem: SDDI}) in the proof of our main result. One can find its proof in Lemma 3.1 in \cite{ha2018uniform}.
\begin{lemma}\label{lem: SDDI}
    Suppose that two nonnegative Lipschitz functions $\mathcal{X}$ and $\mathcal{V}$ satisfy the coupled differential inequalities:
    \begin{equation}\label{eq:sddi}
        \left| \frac{d\mathcal{X}}{dt}\right| \le \mathcal{V}, \quad\frac{d\mathcal{V}}{dt} \le -\alpha\mathcal{V} + \gamma e^{-\alpha t}\mathcal{X}, \quad a.e. \,t>0,
    \end{equation}
    where $\alpha$ and $\gamma$ are positive constants. Then, $\mathcal{X}$ and $\mathcal{V}$ satisfy the uniform bound and decay estimates:
    \begin{equation*}
        \mathcal{X}(t) \le \frac{2M}{\alpha} (\mathcal{X}(0) + \mathcal{V}(0)), \quad \mathcal{V}(t) \le M (\mathcal{X}(0) + \mathcal{V}(0)) e^{-\frac{\alpha t}{2}}, \quad t\ge 0,
    \end{equation*}
    where $M$ is given by
    \begin{equation*}
        M:= \max \left\{1,\frac{2\gamma}{\alpha e}\right\} + \frac{8\gamma}{\alpha^3 e^3}.
    \end{equation*}
\end{lemma}

\subsection{The mean-field limit of the Cucker-Smale model}

The corresponding evolution for the distribution function $\Tilde{f}(x, v, t)$ was first derived by Tadmor, etc. in \cite{Ha2008particle} by the BBGKY hierarchy (e.g. \cite{benedetto1997kinetic,degond2008continuum}), which follows
\begin{equation}\label{eq: mfl for cs}
    \partial_t \Tilde{f} + v\cdot \nabla_x \Tilde{f}+ \nabla_v \cdot [\xi[\Tilde{f}]\Tilde{f}]=0,
\end{equation}
where the operator $\xi$ is defined by
\begin{equation*}
    \xi[f](x,v,t) = \kappa \int\int \psi(|x-y|)(w-v)f(y,w,t)\, dwdy.
\end{equation*}
It is clear that 
\begin{equation}\label{eq:tildevf}
    \frac{\partial}{\partial t}\int \int v\Tilde{f}dxdv =0.
\end{equation}
Similar with energy-decreasing in \eqref{eq:CSproperty} in the microscopic scale, the mean-field limit has the following property: 
\begin{lemma}[\cite{piccoli2015control}, Theorem 3.1]\label{lem: supp} 
Let $f_0$ be compactly supported in $\mathbb{R}^{2d}$ with $\Bar{x} = \int x f_0 (x,v) dxdv$ and $\Bar{v} = \int v f_0 (x,v) dxdv.$ Define $C_{x0}:=\inf \{X\ge0\, | \,\text{supp}(f_0)\subset B(\Bar{x},X) \times \mathbb{R}^d \}$ and $C_v:=\inf \{V\ge0\, | \,  \text{supp}(f_0)\subset \mathbb{R}^d \times B(\Bar{v},V)\},$ where $B(\Bar{x},X)$ designates the ball of center $\Bar{x}$ and radius $X$ in $\mathbb{R}^d.$ Let $\Tilde{f}_t$ be the unique solution of  \eqref{eq: mfl for cs} with $\Tilde{f}(0) = f_0.$ If the initial data satisfies
\begin{equation}\label{eq: C_v}
    C_v < \kappa\int_{C_{x0}}^\infty \psi (2s) ds,
\end{equation}
then there exists $C_x>0$ such that 
    \begin{equation}\label{eq:2.11}
        \text{supp}[\Tilde{f}_t]\subset B(\Bar{x} + t\Bar{v}, C_x) \times B(\Bar{v}, C_v e^{-\kappa\psi(2C_x) t}).
    \end{equation}
\end{lemma}
\begin{remark}
    Under  assumption \eqref{ass:psi} of $\psi,$  one has $\psi(2C_x)\ge \psi_0.$ Actually, the assumption of the lower bound $\psi_0$ is for the convenience of the proof and can be relaxed if $\eqref{eq: C_v}$ is satisfied by both discrete and continuous systems, guaranteeing the flocking in the Cucker-Smale model.
\end{remark}

\subsection{The Wasserstein distance}

To introduce the topology of the space of probability measures, here we recall the definitions of the Wasserstein space \cite{villani2009optimal}. For a metric space 
 $(\mathbb{R}^d,\rho)$ and $q\ge1,$ define the Wasserstein space of order $q$ by
\begin{equation*}
    \mathcal{P}_{q,\rho} (\mathbb{R}^d):= \left\{ \mu \in \mathcal{P}(\mathbb{R^d}); \int_{\mathbb{R}^d} \rho(x,0)^q \nu (dx)< \infty\right\},
\end{equation*}
endowed with the Wasserstein metric 
\begin{equation*}
    W_{q,\rho}(\mu,\nu) := \left( \inf\limits_{\pi \in \Pi (\mu,\nu)} \left\{ \int_{\mathbb{R}^d \times \mathbb{R}^d} \rho(x,y)^q \pi (dx,dy)\right\}\right)^{\frac{1}{q}},
\end{equation*}
where $\mu,\nu \in \mathcal{P}_q (\mathbb{R}^d)$ and $\Pi (\mu,\nu)$ is the set of probability measures on $\mathbb{R}^d \times \mathbb{R}^d$ with marginals $\mu$ and $\nu$ respectively. The $W_q$ metric, or distance, induces a kind of weak topology and measures the closeness between distributions. We use the notation $W_{q,p}$ when $\rho(x,y):= \|x-y \|_p$ and typically abbreviate $W_q = W_{q,2}.$ 
For convenience in notation, we simplify $\|\cdot\|$ as $|\cdot|$ without ambiguity.

Recall the Jordan decomposition for a signed measure $|\mu| := \mu^+ + \mu^-$. One can control the $W_q$ distance by the weighted total variation by the following lemma (see e.g. Theorem 6.15 in \cite{villani2009optimal}), which helps the proof of our main theorems, i.e., Theorem \ref{thm: 1} and Theorem \ref{thm: 2}, in Section \ref{sec: mainthm}.
\begin{lemma}[The Wasserstein distance is controlled by the weighted total variation]\label{lem: TV}
    Let $\mu$ and $\nu$ be two probability measures on a Polish space $(\mathfrak{X}, d).$ Let $q\in [1,\infty)$ and $x_0 \in \mathfrak{X}.$ Then 
    \begin{equation*}
        W_q (\mu,\nu) \le 2^{1-\frac{1}{q}}\left( \int d(x_0,x)^q d |\mu-\nu|(x) \right)^{\frac{1}{q}}.
    \end{equation*}
\end{lemma}

\subsection{The Random Batch method and its mean-field dynamics}

In each sub-time interval $[t_{k-1},t_k),$ we set $[i]_k$ as the batch containing the particle $i.$ The RBM-(approximated) system of \eqref{eq: cs} becomes:
\begin{equation} \label{eq: the RBMcs}
    \left\{\begin{aligned}
        \partial_t X_i^R(t) =& V_i^R(t),\quad i=1,\cdots, N,\\
        \partial_t V_i^R(t) =& \frac{\kappa}{p-1}\sum\limits_{j\in[i]_k} 
        \psi(|X_j^R(t)-X_i^R(t)|)(V_j^R(t)-V_i^R(t)),
    \end{aligned}\right.
\end{equation}
with initial data $(X_i^R(0),V_i^R(0)) = (X_i^{in},V_i^{in}),$ $ i = 1,\cdots,N.$ See Algorithm \ref{algo: the RBM}.

\begin{algorithm}\label{algo: the RBM}
\caption{ The RBM for \eqref{eq: cs}}
\For{$k=1$ to $T/\tau$}
    {Divide $\{1,\cdots, N\}$ into $n = N/p$ batches randomly\;
    
    \For{each batch $\mathcal{C}_q$}
        {
        Update $(X^R_i,V^R_i)$ in $\mathcal{C}_q$ by solving 
        \begin{equation*}
            \left\{\begin{aligned}
                \partial_t X_i^R(t) =& V_i^R(t),\, i=1,\cdots, N,\\
                \partial_t V_i^R(t) =& \frac{\kappa}{p-1}\sum\limits_{j\in\mathcal{C}_q}
                \psi(|X_j^R(t)-X_i^R(t)|)(V_j^R(t)-V_i^R(t)),
            \end{aligned}\right.
        \end{equation*}
        for $t\in [t_{k-1},t_k).$ 
        }
    }
\end{algorithm}


In the $N\to\infty$ limit, the $N$-particle system is then reduced to a $p$-particle
subsystem described by the following system for $t\in [t_{k},t_{k+1}):$
\begin{equation} \label{eq: the RBM limit discrete}
    \left\{\begin{aligned}
        \partial_t X_i(t) =& V_i(t),\quad i=1,\cdots, N,\\
        \partial_t V_i(t) =& \frac{\kappa}{p-1}\sum\limits_{j=1}^p 
        \psi(|X_j(t)-X_i(t)|)(V_j(t)-V_i(t)),
    \end{aligned}\right.
\end{equation}
with $Z_i:=(X_i,V_i)$ and $Z_i (t_k)$ being \textit{i.i.d} drawn from $f(\cdot; t_k).$ Here $f(\cdot; t_k)$ is the law of $Z_1(t_k^{-})$ for $k\ge 1$ and $f(\cdot;0)$ equals the initial distribution $f_0 .$ We impose $Z_1 (t_k^-) = Z_1 (t_k^+).$  Except the individual particle 1, the rest $p-1$ particles are sampled from an infinite pool of independent particles from particle $1$ at each time step $t_k$. This process becomes the mean-field limit model of the RBM system \eqref{eq: the RBMcs}, and one may write the following mean-field limit for the RBM in terms of the probability distribution as shown in Algorithm \ref{algo: the RBM mfl}, while \eqref{eq: the RBM limit discrete} becomes its microscopic description.

\begin{algorithm}\label{algo: the RBM mfl}
\caption{Mean-field Dynamics of the RBM \eqref{eq: the RBM limit discrete}}\label{algorithm}
From $t_k$ to $t_{k+1}$ the distribution function $f_{t_k}$ will be transformed into $f_{t_{k+1}} := \mathcal{Q}_{\infty} (f_{t_{k}})$ as follows\;
Let $f^{(p)}(\cdots;t_k):=f_{t_k}^{\otimes p}$ be a probability measure on $\mathbb{R}^{2dp}$\;
Evolve $f^{(p)}$ by time $\tau$ according to
\begin{equation}\label{the RBM Liouville}
     \partial_t f^{(p)} + \sum\limits_{i=1}^p \nabla_{x_i} \cdot (v_i f^{(p)}) +\sum\limits_{i=1}^p \nabla_{v_i} \cdot \left(\xi_i  f^{(p)} \right)=0,
\end{equation}
     where
 \begin{equation*}
     \xi_i  := \frac{\kappa}{p-1} \sum\limits_{j=1}^p\psi (|x_j-x_i|)(v_j -v_i).
 \end{equation*}
Set $\mathcal{Q}_{\infty}(f_{t_k}) : = \int _{\mathbb{R}^{2d(p-1)}} f^{(p)}(\cdots; t_{k+1}^-) dx_2 \cdots dx_p dv_2\cdots dv_p .$

\end{algorithm}

\section{The Main Result}\label{sec:main}

In this paper, we show the conservation law and the mean-field limit of the Random-Batch Cucker-Smale model. First, we show the conservation of momentum and dissipation of kinetic energy of the  mean-field limit system of RBM in Section \ref{sec:prop}. Then, we analyze the procedure (b) and (c) in Section \ref{sec: mainthm}. 

Assume that the initial data of Algorithm \ref{algo: the RBM} and Algorithm \ref{algo: the RBM mfl} obey the same law $f_0(x,v)dxdv.$ For notational convenience, recall $\Bar{x}:= \int x f_0(x,v)$ and $\Bar{v}:=\int vf_0 (x,v);$ and set $z_i:=(x_i,v_i)$ for each $i$ and 
\[
f(x,v;t):= \int_{\mathbb{R}^{2d(p-1)}} f^{(p)}(\cdots ; t^-)dx_2 \cdots dx_p dv_2\cdots dv_p = f_t(x,v).
\]
Note that $f$ is defined on $\mathbb{R}^{2d}$ in Algorithm \ref{algo: the RBM mfl} and $f^{(p)}$ is defined on $\mathbb{R}^{2dp}.$ Also, the definition of $f$ is consistent with that in Algorithm \ref{algo: the RBM mfl} at $t=t_k.$

In the following, we suppose that $f_0$ is compactly supported and then there exists a constant $R>0$ such that supp$[f_0]\subset B(0,R),$ where $B(0,R)$ is a ball with radius $R$ centered in the origin. Without loss of generality, we set $\Bar{v}=0$ as a standardization of the initial momentum.

\subsection{Basic characteristics of the mean-field limit system of RBM}\label{sec:prop}

Here, we show that the two basic characteristics: conservation of momentum and dissipation of kinetic energy, as observed in the Cucker-Smale model \eqref{eq: cs} and the RBM system \eqref{eq: the RBMcs}, also holds for the RBM mean-field system described in Algorithm \ref{algo: the RBM mfl}.

\begin{proposition}[Conservation of momentum] \label{prop: M1v}
For $t\in (t_{k},t_{k+1}),$ it holds
    \begin{equation*}
        \frac{\partial}{\partial t} \int_{\mathbb{R}^{2d}} v f(x,v;t) dx dv = 0.
    \end{equation*}
In particular, since $\Bar{v}=0,$ $\int vf(x,v;t) = 0$ for any $t\in [0,T].$
\end{proposition}
\begin{proof}
By the definition of $f,$ one has
\begin{equation*}
    \begin{aligned}
        \frac{\partial}{\partial t}\int_{\mathbb{R}^{2d}} vf(x,v;t)dxdv = &
        \int_{\mathbb{R}^{2d}} \int_{\mathbb{R}^{2d(p-1)}} v \frac{\partial}{\partial t} f^{(p)} (x,\cdot, v, \cdot;t^-)dxdx_2\cdots dx_p dvdv_2\cdots dv_p\\
        = & \int_{\mathbb{R}^{2dp}} v_1 \left[ -\sum\limits_{i=1}^p \nabla_{x_i}\cdot v_i f^{(p)}(\cdots ; t) - \sum\limits_{i=1}^p \nabla_{v_i} \cdot \xi_i f^{(p)}(\cdots ;t)\right]dz_1\cdots dz_p\\
        = & \frac{\kappa}{p-1}\sum\limits_{j=2}^p\int_{\mathbb{R}^{2dp}} \psi (|x_j - x_1|)(v_j-v_1) f^{(p)}_t(z_1,\cdots ,z_p)dz_1\cdots dz_p\\
        = & 0.
    \end{aligned}
\end{equation*}
The last equality is derived from the fact that $\psi (|x_j - x_1|)(v_j-v_1)f^{(p)}_t$ is an odd function on $(z_j,z_1).$
Note that $f(x,v;t_k^-) = f(x,v;t_k) = f(x,v;t_k^+).$ Then by induction one has $\int vf(x,v;t) = \Bar{v}.$
\end{proof}

\begin{proposition}[Dissipation of kinetic energy] \label{prop: M2v}
For any $t\in (t_{k},t_{k+1}),$
    \begin{equation*}
        \frac{\partial}{\partial t} \int_{\mathbb{R}^{2d}} |v|^2 f(x,v;t)dxdv \le 0.
    \end{equation*}
In particular, $M_2^v f_t \le M_2^vf_0,$ where $M_2^v f_t:=\int_{\mathbb{R}^{2d}} |v|^2 f(x,v;t)dxdv.$ Moreover, since $0<\psi_0\le \psi,$ it holds
    \begin{equation*}
        M_2^v f_{k+1} \le M_2^v f_0\left(e^{-2\kappa \psi_0 \tau} + \frac{8\kappa^2 \psi_0}{p-1}\tau^2 \right)^{k} .
    \end{equation*}
\end{proposition}

    \begin{proof}
    For any $t\in (t_{k},t_{k+1}),$
        \begin{equation*}
            \begin{aligned}
                \frac{\partial}{\partial t} M_2^v f_t = & \int_{\mathbb{R}^{2dp}} |v_1|^2 \left( -\sum\limits_{i=1}^p \nabla_{x_i}\cdot (v_i f^{(p)}_t) - \sum\limits_{i=1}^p \nabla_{v_i}\cdot (\xi_i f^{(p)}_t) \right)\\
                =& \int_{\mathbb{R}^{2dp}} \frac{\kappa}{p-1}\sum\limits_{j=1}^p\psi(|x_1-x_j|)(v_j-v_1)\cdot 2v_1 f^{(p)}_t\\
                =& -\int_{\mathbb{R}^{2dp}} \frac{\kappa}{p-1}\sum\limits_{j=1}^p\psi(|x_1-x_j|)|v_j-v_1|^2 f^{(p)}_t\\
                \le& 0.
            \end{aligned}
        \end{equation*}
        The last equality is derived by the symmetry of $\psi$ and $f_t^{(p)}.$ By the continuity of $f,$ one has $M_2^v f_t \le M_2^vf_0.$ Since $\psi$ has a positive lower bound $\psi_0,$ furthermore, it holds
        \begin{equation}\label{eq:3.09}
            \begin{aligned}
                \frac{\partial}{\partial t} M_2^v f_t \le & -\int_{\mathbb{R}^{2dp}} \frac{\kappa}{p-1}\sum\limits _{j=1}^p \psi_0 |v_j-v_1|^2 f^{(p)}_t\\
                =& -2\kappa \psi_0 \int_{\mathbb{R}^{2d}} |v|^2 f_t(x,v) dxdv + 2\kappa\psi_0 \int_{\mathbb{R}^{2dp}} v_1\cdot v_2 f^{(p)}_t,
            \end{aligned}
        \end{equation}
        by the symmetry of $f^{(p)}_t$. For the second term in \eqref{eq:3.09}, one has
        \begin{equation}\label{eq:3.10}
            \begin{aligned}
                \frac{\partial}{\partial t}& \int_{\mathbb{R}^{2dp}} v_1 \cdot v_2 f^{(p)}_t\\
                = & \kappa\int_{\mathbb{R}^{2dp}} \frac{v_2}{p-1}\cdot\sum\limits_{k\ne 1}^p \psi(|x_1-x_k|)(v_k-v_1)f^{(p)}_t + \frac{v_1}{p-1}\cdot\sum\limits_{k\ne 2}^p \psi(|x_2-x_k|)(v_k-v_2) f^{(p)}_t\\
                = &2\kappa\int_{\mathbb{R}^{2dp}} \frac{v_2}{p-1}\cdot\sum\limits_{k\ne 1}^p \psi(|x_1-x_k|)(v_k-v_1)f^{(p)}_t.
            \end{aligned}
        \end{equation}
        For the last integration in \eqref{eq:3.10}, by the indistinguishability of the particles, one has  
        \begin{equation*}
            \begin{aligned}
                &\int_{\mathbb{R}^{2dp}} v_2\cdot\sum\limits_{k\ne 1}^p \psi(|x_1-x_k|)(v_k-v_1)f^{(p)}_t\\
                &= \int_{\mathbb{R}^{2dp}} v_2\cdot\sum\limits_{k\ne 1,2}^p \psi(|x_1-x_k|)(v_1-v_k)f^{(p)}_t + \int_{\mathbb{R}^{2dp}} v_2\cdot \psi(|x_1-x_2|)(v_2-v_1)f^{(p)}_t\\
                &= \int_{\mathbb{R}^{2dp}} v_2\cdot\sum\limits_{k\ne 1}^p \psi(|x_1-x_k|)(v_1-v_k)f^{(p)}_t 
                +\int_{\mathbb{R}^{2dp}} v_2\cdot \psi(|x_1-x_2|)(v_2-v_1)f^{(p)}_t\\
                &+ \int_{\mathbb{R}^{2dp}} v_2\cdot \psi(|x_1-x_2|)(v_2-v_1)f^{(p)}_t\\
                &= -\int_{\mathbb{R}^{2dp}} v_2\cdot\sum\limits_{k\ne 1}^p \psi(|x_1-x_k|)(v_k-v_1)f^{(p)}_t +\int_{\mathbb{R}^{2dp}} |v_2-v_1|^2 \psi(|x_1-x_2|) f^{(p)}_t.
            \end{aligned}
        \end{equation*}
        Hence, recalling $\psi\le 1,$ it holds
        $$\int_{\mathbb{R}^{2dp}} v_2\cdot\sum\limits_{k\ne 1}^p \psi(|x_1-x_k|)(v_k-v_1)f^{(p)}_t
        = \frac{1}{2}\int_{\mathbb{R}^{2dp}} |v_2-v_1|^2 \psi(|x_1-x_2|)f^{(p)}_t\le 2 M_2^v f_t.$$
        Then it yields from \eqref{eq:3.10} that 
        \begin{equation}\label{eq:3.12}
            \frac{\partial}{\partial t} \int_{\mathbb{R}^{2dp}} v_1 \cdot v_2 f^{(p)}_t\le \frac{4\kappa}{p-1} M_2^vf_{t_{k-1}}.
        \end{equation}
        By combining \eqref{eq:3.10} and \eqref{eq:3.12}, it holds
        \begin{equation*}
        \begin{aligned}
            \frac{\partial}{\partial t}M_2^vf_t &\le -2\kappa\psi_0 M_2^vf_t + 2\kappa\psi_0\frac{4\kappa\tau}{p-1}M_2^vf_{t_{k-1}} +  2\kappa\psi_0\int v_1 \cdot v_2 f_{t_{k-1}}^{(p)} \\
            &=  -2\kappa\psi_0 M_2^vf_t + 2\kappa\psi_0\frac{4\kappa\tau}{p-1}M_2^vf_{t_{k-1}}  \\
            &\le  -2\kappa\psi_0M_2^v f_t + \frac{8\kappa^2\psi_0}{p-1}\tau M_2^vf_{t_{k-1}}.
        \end{aligned}
        \end{equation*}
        By Gr\"onwall's inequality, one has
        $$ M_2^v f_{k} \le M_2^v f_{t_{k-1}}\left(e^{-2\kappa \psi_0 \tau} + \frac{8\kappa^2\psi_0}{p-1}\tau^2 \right) \le M_2^v f_0\left(e^{-2\kappa \psi_0 \tau} + \frac{8\kappa^2 \psi_0}{p-1}\tau^2 \right)^{k} .$$
    \end{proof}


\subsection{The main theorems}\label{sec: mainthm}

First we discuss the Wasserstein distance between the distribution of the RBM of the $N$-particle system \eqref{eq: the RBMcs} and the distribution of the $N\to \infty$ system \eqref{eq: the RBM limit discrete}. 

Define the operator $ \mathcal{Q}_N^{(k)}$ on the probability measure space $\mathcal{P}(\mathbb{R}^{2d})$ as follows. Let $Z_i^R(0)$ be \textit{i.i.d} drawn from $f_0.$ Corresponding to Algorithm \ref{algo: the RBM}, we define 
$$ \mathcal{Q}_N^{(k)}(f_0):= \text{Law}(Z_1^R(t_k)), $$
where $\text{Law}(Z_1^R)$ means the law of $Z_1^R.$  Conditioning on a specific sequence of random batches, the particles are not exchangeable. However, when considering the mixture of all possible sequences of random batches, the laws of $Z_i^R(t_k)$ are identical.

In addition, we define $ \mathcal{Q}_\infty$ as one time step evolution in Algorithm \ref{algo: the RBM mfl}. Thus one arrives at 
$$\mathcal{Q}_\infty^k (f_0) = \mathcal{Q}_\infty \circ \cdots \circ \mathcal{Q}_\infty(f_0), \text{ ($k$ copies)}, $$
which is expected to be the mean field limit of the RBM after $k$ steps. Note that the $\mathcal{Q}_\infty$ dynamics can fully determine the probability transition, knowing the marginal distribution of $Z_1,$ while knowing only the marginal distribution is not enough for the dynamics of $\mathcal{Q}_N^{(k)}$. The joint distribution must be known in the latter case.

Theorem \ref{thm: 1} gives an estimate between $\mathcal{Q}_\infty^k (f_0)$ and $\mathcal{Q}_N^{(k)}(f_0)$ on the $W_q$ distance. In this theorem, we introduce a quantity $\epsilon_k\in \mathbb{R},$ which was first introduced in \cite{jin2022mean}. Plainly speaking, the $\epsilon_k$ means the probability that the particle $i$ is not ``clean", where a particle $i$ is clean at $t_k^-$ means its batchmates at $t\le t_k$ were mutually independent and independent to the particle $i$ when they interacted. For better organization, we give the definition of "clean"   and more details in Section \ref{subsec: def of ek}.

\begin{theorem}\label{thm: 1}
    For $f_0$ with the compact support $\textit{supp} [f_0],$ it holds that 
    \begin{equation}\label{eq: thm3.1}
        W_q(\mathcal{Q}_\infty^k (f_0),\mathcal{Q}_N^{(k)} (f_0))\le C (1+t_k) \epsilon_k^{1/q},
    \end{equation}
    with $q \in [1,\infty)$ and $C=C(\kappa,\psi_0,R).$ For fixed $k,$ it holds
    \begin{equation}\label{eq: thm3.2}
        \lim\limits_{N\to \infty} \epsilon_{k} =0,
    \end{equation}
    with $\epsilon_{k}\le\mathcal{O}(N^{-1}).$
\end{theorem}

    


    Then, we consider the limit dynamics given by the operator $\mathcal{Q}_\infty$ and how $\mathcal{Q}_\infty^n (f_0)$ approximates the dynamics of the Fokker-Planck equation \eqref{eq: mfl for cs}, i.e. the procedure (c) in \eqref{graph}. We prove that the $W_q$ distance can be controlled by $\mathcal{\tau}$ with a smooth initial distribution $f_0$ with a compact support uniformly in time. 
    
    \begin{theorem} \label{thm: 2}
        Let $\Tilde{f}$ be a solution to the Fokker-Planck  equation \eqref{eq: mfl for cs} with initial condition $\Tilde{f}(0) = f_0.$ 
        It holds
        the uniform-in-time estimate
        \begin{equation}\label{eq: thm2}
            \sup_{n} W_q (\mathcal{Q}_\infty^n (f_0), \Tilde{f}(n\tau))\le C\tau,
        \end{equation}
        with $q\in[1,\infty)$and $C=C(\kappa,\psi_0,R).$
    \end{theorem}


\subsection{Some helpful discussions}\label{sec:discussion}

In addition, we give some  discussions on Lemma \ref{lem: 10} that will be used for the proof of Theorem \ref{thm: 2}, and on the commutation of \eqref{graph}.

\subsubsection{Approximation using PDE analysis}

Similar with [\cite{jin2022mean}, Lemma 4.4], one can obtain an $\mathcal{O}(\tau^{\frac{2}{q}})$ bound on the $W_q$ distance between $\mathcal{Q}_\infty (\Tilde{f}_{t_k})$ and $\Tilde{S}(\Tilde{f}_{t_k})$ in Proposition \ref{lem: 10old}. Here $\Tilde{S}(\Tilde{f}_{t_k})$ denotes the law of \eqref{eq: tempsys} in Section \ref{subsec: auxiliary sys}.

\begin{proposition}\label{lem: 10old}
     Let $\Tilde{f}$ be a solution to the Fokker-Planck  equation \eqref{eq: mfl for cs} with initial condition $\Tilde{f}(0) = f_0 \in (C^2\cap W^{2,\infty})(\mathbb{R}^{2d}).$ Suppose that the interaction kernel $\psi$ satisfies $$\sup\limits_{r\ge 0} (|\psi^\prime(r)| + |\psi^{\prime\prime}(r)|) \le \epsilon_\psi$$ with a constant $\epsilon_\psi.$ It holds
    \begin{equation}\label{eq:lmm10old}
        W_q (\mathcal{Q}_\infty (\Tilde{f}_{t_k}),\Tilde{S}(\Tilde{f}_{t_k})) \le C(\kappa,\epsilon_\psi , R, d) \tau^{\frac{2}{q}}, \quad q\in [1,\infty).
    \end{equation}
\end{proposition}
The proof relies on the PDE analysis on the evolution equations on $\mathcal{Q}_\infty (\Tilde{f}_{t_k})$ and $\Tilde{S}(\Tilde{f}_{t_k}),$ using the corresponding operator groups. There are two limitations on this method. One is that more regular assumptions on both the interaction kernel and initial data are needed to guarantee the existence of a unique classical solution $\Tilde{f}_t \in C^2([0,T)\times \mathbb{R}^{2d})$ for \eqref{eq: mfl for cs}. In addition, if one uses the estimate obtained by Lemma \ref{lem: TV} with the total variance bound, the order of $\tau$ is $\frac{2}{q}.$ 
A detailed proof of Proposition \ref{lem: 10old} is given in Appendix \ref{app:semigroup}.

\subsubsection{The mean-field limit of the Cucker-Smale system} \label{subsec: meanCS}

Up to now, we have given a brief analysis of the path ((b),(c)) in the lower right corner of \eqref{graph}. Here, we discuss the upper left corner path in the sense of the Wasserstein distance as well. Denote by 
$$
\mu_t^N:=\sum\limits_{i=1}^N \frac{1}{N}\delta_{z_i(t)}$$ the empirical measure and $\mu_t^R$ the empirical measure derived by the RBM with $\mu_0^R = \mu_0^N.$

\begin{proposition} 
    For the path from (a) to (d) in \eqref{graph}, it holds 
    \begin{equation}
        \mathbb{E} W_2 (\Tilde{f}_t,\mu_t^R)     \le C(\kappa, R,\psi)\sqrt{\frac{\tau}{p-1}} +
        C(R)\cdot\left\{\begin{aligned}
            &N^{-\frac{1}{4}} \quad &\text{if } d<4,\\
        &N^{-\frac{1}{4}}\sqrt{\log(1+N)} \quad &\text{if } d=4,\\
        &N^{-\frac{1}{d}} \quad &\text{if } d>4.
        \end{aligned}\right.
    \end{equation}
\end{proposition}
\begin{proof}
    We first consider the procedure (a). For the first $\tau\to 0$ limit, by the definition of the $W_q$ distance, one has
    \begin{equation*}
        W_2^2 (\mu_t^R, \mu_t^N) 
        \le \frac{1}{N} \sum\limits_{i=1}^N (|X_i^R - X_i|^2 + |V_i^R - V_i|^2).
    \end{equation*}
    Based on [\cite{ha2021uniform}, Theorem 3.2], one has
    \begin{equation*}
        \mathbb{E}\left(\frac{1}{N}\sum\limits_{i-1}^N |V_i^R (t) - V_i (t)|^2\right)
        \le C\left(\frac{1}{p-1}-\frac{1}{N-1}\right)\tau + C\tau^2 + C(1+\tau) \exp{(-\kappa\psi_0 t)},
    \end{equation*}
    for some constant C depending on $\kappa, R$ and $\psi.$
    Note that
    \begin{equation*}
    \frac{d}{dt} \sum\limits_{i=1}^N |X_i^R - X_i|^2 \le
    2\sqrt{\sum\limits_{i=1}^N |X_i^R - X_i|^2} \sqrt{\sum\limits_{i=1}^N |V_i^R - V_i|^2},
    \end{equation*}
    by the Cauchy-Schwarz inequality and
    \begin{equation*}
        \sqrt{\sum\limits_{i=1}^N |X_i^R - X_i|^2} \le \sqrt{\sum\limits_{i=1}^N |V_i^R - V_i|^2} T.
    \end{equation*}
    It thus holds
    \begin{multline*}
    \mathbb{E} W_2 (\mu_t^N,\mu_t^R) \le  \mathbb{E} \frac{1}{\sqrt{N}}\sqrt{\sum\limits_{i=1}^N |X_i^R - X_i|^2 + \sum\limits_{i=1}^N |V_i^R - V_i|^2} \\ 
    \le  \sqrt{C \tau \left(\frac{1}{p-1}-\frac{1}{N-1}\right)}(1+T).
    \end{multline*}.

    Next, for the procedure (d), by the estimate of [\cite{ha2018uniform}, Corollary 1], one has $W_q (\mu_t^N,\Tilde{f}_t) \le C(\kappa,R)W_q(\mu_0^N,\Tilde{f}_0).$ 
    In addition, according to the convergence rate of the empirical measure in [\cite{fournier2015rate}, Theorem 1], with $f_0$ compactly supported, one has 
    $$
    \mathbb{E} W_q^q (f_0,\mu_0^N) \le C(R)\cdot \left\{
    \begin{aligned}
        &N^{-\frac{1}{2}} \quad &\text{if } q>\frac{d}{2},\\
        &N^{-\frac{1}{2}}\log(1+N) \quad &\text{if } q=\frac{d}{2},\\
        &N^{-\frac{q}{d}} \quad &\text{if } q<\frac{d}{2}.
    \end{aligned}\right.
    $$

    Therefore, under the same assumption of Theorem \ref{thm: 1} and Theorem \ref{thm: 2}, for the path from (a) to (d), one has 
    \begin{equation*}
    \begin{aligned}
         \mathbb{E} W_2 (\Tilde{f}_t,\mu_t^R) &\le  \mathbb{E} \left[W_2 (\mu_t^N,\mu_t^R) + W_2 (\Tilde{f}_t,\mu_t^N)\right]    
         \le C(\kappa,R,\psi)\sqrt{\frac{\tau}{p-1}} + C(\kappa,R) \mathbb{E} W_2 (f_0,\mu_0^N)\\
         &\le C(\kappa,R,\psi)\sqrt{\frac{\tau}{p-1}} + C(\kappa,R)\left(\mathbb{E} W_2^2 (f_0,\mu_0^N)\right)^{\frac{1}{2}}.
    \end{aligned}
    \end{equation*}
\end{proof}


\section{Proof of Theorem \ref{thm: 1}}\label{sec: mainpf1}

    In this section, we prove Theorem \ref{thm: 1} by an observation in combinatorics. 
    Recalling the critical quantity $\epsilon_k$ we mentioned above Theorem \ref{thm: 1}, we first give the Lemmas \ref{lem: 2ek}-\ref{lem: thm3.2}, for the proof of Theorem \ref{thm: 1}, and show the mathematical definition of $\epsilon_k$ in Section \ref{subsec: def of ek}.
    The proof of Lemmas \ref{lem: 2ek}-\ref{lem: thm3.2} are shown in Sections \ref{subsec: lem4}-\ref{subsec: lem5} respectively.
    Finally, Theorem \ref{thm: 1} will be proved in Section \ref{subsec: pfthm1} by combining Lemma \ref{lem: 2ek}, Lemma \ref{lem: thm3.2} and an estimate on the diameters of the compact support in space and velocity, shown in Lemma \ref{ineq: coarseD}.

    
    \begin{lemma}\label{lem: 2ek}
        Under the assumption and notation of Theorem \ref{thm: 1}, it holds
            \begin{equation*}
                |\mathcal{Q}_\infty^k (f_0) - \mathcal{Q}_N^{(k)} (f_0)|_{TV} \le 2\epsilon_k,
            \end{equation*}
        where $\epsilon_k$ is a positive constant  with respect to $N.$
    \end{lemma}

    \begin{lemma}[\cite{jin2022mean}, Theorem 3.2]\label{lem: thm3.2}
        For any fixed $k,$ with the assumption and notation of Theorem \ref{thm: 1}, it holds that $\epsilon_k \le \mathcal{O}(N^{-1}).$
    \end{lemma}

\subsection{Prerequisites}\label{subsec: def of ek}
    To give a detailed explanation of $\epsilon_k$ in Theorem \ref{thm: 1}, we introduce the concept ``clean'' in Definition \ref{def:clear}. First, for each particle $i$, we define a sequence of lists $\{L_i^{(k)}\}$ associated with $i$, given as follows:
    \begin{enumerate}
        \item $L_i^{(0)} = \{i\};$
        \item For $k\ge 1,$ let $\mathcal{C}_q^{(k-1)}$ be the batch that particle $i$ stays in for $t\in[t_{k-1},t_k),$ then
        \begin{equation*}
            L_i^{(k)} = \cup_{j\in \mathcal{C}_q^{(k-1)}} L_j^{(k-1)}.
        \end{equation*}
    \end{enumerate}
    Here, $L_i^{(k)}$ can be viewed as the particles that have impacted $i$ for $t < t_k.$ Clearly, a particle $i_1\in L_i^{(k)}$ might not have been a batchmate of $i.$ It could have been a batchmate of $i_2,$ and then $i_2$ was a batchmate of $i$ at some time. The important observation is that if $L_i^{(k)}$ and $L_j^{(k)}$ do not intersect for a given sequence of random batches, then particles $i$ and $j$ are independent at $t_k^-.$ Note that we are not claiming all particles in $L_j^{(k)}$ are independent of those in $L_j^{(k)}$ at $t_k^-.$ In fact, it is possible that some $i_1 \in L_i^{(k)}$ and $j_1 \in L_j^{(k)}$ are in the same batch on $[t_{k-1}, t_k).$ However, $i_1$ and $j_1$ must be independent at the times when they were added to the batches that eventually impact $i,j$ at $t_k^-.$ Then we are motivated to define the following.
    
    \begin{definition}\label{def:clear}
        We say the particle $i$ is clean on $[t_k,t_{k+1})$ if the batch $\mathcal{C}_q^{(k)}$ that contains $i$ at $t_k^+$ satisfies the following:
        \begin{enumerate}[(1)]
            \item any $j\in\mathcal{C}_q^{(k)}$ is clean at $t_k^-$;
            \item any $j,\ell \in\mathcal{C}_q^{(k)}$ with $j\ne \ell,$ $L_j^{(k)}$ and $L_\ell^{(k)}$ do not intersect.
        \end{enumerate}
    \end{definition}
    In other words, a particle $i$ is ``clean" at $t_k^-$ if its batchmates at $t<t_k$ were mutually independent and independent to $i$ when they interacted.

    Let $A_k$ denote the set of particles that are clean at $t_k^-.$ Then
    \begin{equation*}
        A_0 = \{1,\cdots, N\}.
    \end{equation*}
    For $k\ge 1,$ one has 
    \begin{multline}
        A_k = \Big\{ i \in A_{k-1}: i\in \mathcal{C}_q^{(k-1)},\, \forall j,\ell \in\mathcal{C}_q^{(k-1)},\, j\ne \ell ,\\
        j\in A_{k-1}, \ell \in A_{k-1},\, L_j^{(k-1)}\cap L_\ell^{(k-1)} = \emptyset \Big\}.
    \end{multline}
    Now we define
    \begin{equation}
        \epsilon_k := \mathbb{P}(1\notin A_k).
    \end{equation}
    Note that by symmetry, $\epsilon_k$ is also the probability that particle $i$ is not clean.

\subsection{Proof of Lemma \ref{lem: 2ek}}\label{subsec: lem4}

    We first introduce Lemma \ref{lem:app 3.2}.
     \begin{lemma}\label{lem:app 3.2}
        Consider a fixed sequence of divisions of random batches $\{\mathcal{C}^{(\ell)}\}_{\ell\le k-1}.$
        \begin{itemize}
            \item It holds that
                \begin{equation*}
                    |L_i^{(k)}|\le p^k,
                \end{equation*}
                and the particle $i$ is clean at $t_k^-$ if and only if the equality holds.
            \item The distribution of $Z_i$ (Recall $Z_i := (X_i,V_i).$) for a clean particle $i$ at $t_k^-$ is $\mathcal{Q}_\infty^k (f_0).$
        \end{itemize}
        Here the symbol $|A|$ means the cardinality of a set $A.$
    \end{lemma}
    \begin{proof}
        The proof is a straightforward induction. Here, let us just briefly mention the proof of the second claim. For $k=0,$ the statement is trivial. Now, suppose the statement is true for all $k \le m-1.$ We now consider $k = m.$
        
        For the given sequence of random batches $\{\mathcal{C}^{(\ell)}\}_{\ell\le k-1},$ that a particle $i$ is clean at $t_m^-$ means that on $[t_{m-1}, t_m),$ the particles in the batch for $i$ are independent of $i$ at $t_{m-1}.$ By the induction assumption, the distribution of one particle at $t_{m-1}$ is given by $\mathcal{Q}_\infty^{m-1}(f_0).$ By the independence, the joint distribution of them at $t_{m-1}$ is therefore
        \begin{equation*}
            f^{(p)} (\cdots, t_{m-1}) = \mathcal{Q}_\infty^{m-1}(f_0)^{\otimes p}.
        \end{equation*}
        From $t_{m-1}$ to $t_m,$ the evolution of the joint distribution obeys  equation \eqref{the RBM Liouville}. Hence, at $t_m^-,$ the distribution of particle $i$ is given by $\mathcal{Q}_\infty^m(f_0).$ 
    \end{proof}

    Now we begin to prove Lemma \ref{lem: 2ek}.

    \begin{proof}[Proof of Lemma \ref{lem: 2ek}]
        By Lemma \ref{lem:app 3.2}, one has
        \begin{equation}\label{eq:B4}
            \mathcal{Q}_N^{(k)} (f_0) = \mathbb{P}(1\in A_k)\mathcal{Q}_\infty^k (f_0)
            + \mathbb{P}(1\notin A_k)\nu_k,
        \end{equation}
        for some probability measure $\nu_k.$ To see this, consider all possible sequences of random batches. Only the first $k$ divisions of batches (i.e. ones at $t_0,\cdots, t_{k-1}$) will affect the distribution at $t_k.$ This subsequence (the first $k$ divisions) can take only finitely many values, and let $\{c^\ell\}_{\ell\le k-1}$ be such values. Then, for any $E\subset \mathbb{R}^{d}$ that is Borel measurable,
        \begin{equation*}
            \mathcal{Q}_N^{(k)} (f_0)[E] = \sum\limits_{\{\mathcal{C}^\ell = c^\ell , \ell \le k-1\}} \mathbb{P}(\mathcal{C}^\ell = c^\ell , \,\ell \le k-1)
            \mathbb{P}(Z_1 \in E \,|\, \mathcal{C}^\ell = c^\ell , \,\ell \le k-1 ).
        \end{equation*}
        Lemma \ref{lem:app 3.2} tells us if $\{c^\ell \}_{\ell \le k-1}$ is a value such that particle $1$ is clean, then
        \begin{equation*}
            \mathbb{P}(Z_1 \in E \,|\, \mathcal{C}^\ell = c^\ell , \,\ell \le k-1 ) = \mathcal{Q}_\infty^k (f_0)[E].
        \end{equation*}
        Hence,
        \begin{equation*}
            \begin{aligned}
            \mathcal{Q}_N^{(k)} (f_0)[E] &=  \mathbb{P}(1\in A_k)\mathcal{Q}_\infty^k (f_0)[E] \\
            &\quad+ \sum\limits_{\{\mathcal{C}^\ell = c^\ell , \ell \le k-1\},\, 1\notin A_k} \mathbb{P}(\mathcal{C}^\ell = c^\ell , \,\ell \le k-1)
            \mathbb{P}(Z_1 \in E \,|\, \mathcal{C}^\ell = c^\ell , \,\ell \le k-1 )\\
            &= \mathbb{P}(1\in A_k)\mathcal{Q}_\infty^k (f_0)[E] + \mathbb{P}(1\notin A_k)\nu_k (E),
            \end{aligned}
        \end{equation*}
        with
        \begin{equation*}
            \nu_k (E) = \sum\limits_{\{\mathcal{C}^\ell = c^\ell , \ell \le k-1\},\, 1\notin A_k} \frac{\mathbb{P}(\mathcal{C}^\ell = c^\ell , \,\ell \le k-1)}{\mathbb{P}(1\notin A_k)}
            \mathbb{P}(Z_1 \in E \,|\, \mathcal{C}^\ell = c^\ell , \,\ell \le k-1 ).
        \end{equation*}
        Clearly, $\nu_k$ is a convex combination of some conditional marginal distributions of $Z_1,$ each being Law$(Z_1)$ conditioning on a particular sequence of batches for $\{1\notin A_k\}.$ Hence $\nu_k$ is a probability measure. 
        
        By \eqref{eq:B4}, it holds that
        \begin{equation*}
            |\mathcal{Q}_\infty^k(f_0) - \mathcal{Q}_N^{(k)} (f_0)| \le 
            (1-\mathbb{P}(1\in A_k))\mathcal{Q}_\infty^k (f_0) + \mathbb{P}(1\notin A_k)\nu_k
            = \epsilon_k (\mathcal{Q}_\infty^k (f_0) + \nu_k).
        \end{equation*}
        Therefore, the total variation distance between the two measures is controlled by
        \begin{equation*}
            |\mathcal{Q}_\infty^k(f_0) - \mathcal{Q}_N^{(k)} (f_0)|_{TV} \le 
            (1-\mathbb{P}(1\in A_k)) + \mathbb{P}(1\notin A_k) = 2\epsilon_k.
        \end{equation*}
    \end{proof}

\subsection{Proof of Lemma \ref{lem: thm3.2}}\label{subsec: lem5}
    Furthermore, we provide the proof of Lemma \ref{lem: thm3.2} which is a simple variant of [\cite{jin2022mean}, Theorem 3.2].

    \begin{proof}[Proof of Lemma \ref{lem: thm3.2}]
        First of all, clearly, one has
        \begin{equation*}
            \epsilon_0  = 1-1 = 0.
        \end{equation*}
        Now, one can do induction on $k.$ Assume
        \begin{equation*}
            \lim\limits_{N\to \infty} \epsilon_k = 0,
        \end{equation*}
        and then consider the batches for $t_k \to t_{k+1}^-.$ Assume the batch for particle $1$ is $\mathcal{C}_q^{(k)}.$ Denote
        \begin{equation*}
            B_k = \left\{ \forall j, \ell \in \mathcal{C}_q^{(k)},\, j\ne \ell, \, L_j^{(k)}\cap L_\ell^{(k)} = \emptyset \right\}.
        \end{equation*}
        Let $\mathcal{B} := \mathcal{C}_q^{(k)}\backslash \{1\}$ be the set of other particles that share the same batch with particle $1.$ Then, by definition of $A_{k+1},$ it holds
        \begin{multline}
             \mathbb{P} (1\in A_{k+1}) = \sum\limits_{j_1,\cdots, j_{p-1}} \mathbb{P}(\mathcal{B}=\{j_1,\cdots, j_{p-1}\})\times \\
             \mathbb{P}(B_k \cap \{1\in A_k\} \cap_{\ell =1}^{p-1}\{j_\ell \in A_k\}\,\mid \, \mathcal{B}= \{j_1,\cdots, j_{p-1}\}).
        \end{multline}
         Denote $E:= \{\mathcal{B}=\{j_1,\cdots, j_{p-1}\}\},$ where we omit the dependence in $j_\ell,$ $1\le \ell \le p-1$ for notational convenience. Conditioning on $B_k \cap E$ (i.e., provided that the event $B_k \cap E$ happens), whether the particles are clean or not is independent. Hence,
         \begin{multline*}
             \mathbb{P}(B_k \cap \{1\in A_k\} \cap_{\ell =1}^{p-1}\{j_\ell \in A_k\}\,|\, E) \\
             = \mathbb{P}(B_k | E))\mathbb{P}(\{1\in A_k\} \cap_{\ell = 1}^{p-1}\{j_\ell \in A_k\}|E,B_k)\\
             = \mathbb{P}(B_k|E)\Pi_{\ell=1}^p\mathbb{P}(j_\ell\in A_k|E,B_k),
         \end{multline*}
         where one has set $j_p =1.$ Moreover,
         \begin{equation*}
             \mathbb{P}(j_\ell\in A_k|E,B_k) = \frac{\mathbb{P}(\{ j_\ell \in A_k \} \cap E\cap B_k)}{\mathbb{P}(E\cap B_k)} = \frac{\mathbb{P}(\{1 \in A_k \} \cap E\cap B_k)}{\mathbb{P}(E\cap B_k)} = \frac{\mathbb{P}(\{1 \in A_k \} \cap B_k)}{\mathbb{P}(B_k)}.
         \end{equation*}
         The second and the last equalities are due to symmetry. For the last equality, $\mathbb{P}(\{\mathcal{B} = \{j_1\cdots,j_{p-1}\}\}\cap B_k)$ should be equal for all possible $j_1,\cdots,j_{p-1},$ and the same  is true for the numerator. This actually is a kind of independence. Hence, eventually due to the fact
         \begin{equation*}
             \sum\limits_{j_1,\cdots, j_{p-1}} \mathbb{P} (\mathcal{B} = \{j_1,\cdots, j_{p-1}\}) \mathbb{P}(B_k | \mathcal{B} = \{j_1,\cdots, j_{p-1}\}) = \mathbb{P} (B_k),
         \end{equation*}
        one has
        \begin{equation*}
            1-\epsilon_{k+1} = \mathbb{P} (1\in A_{k+1}) \ge
            \mathbb{P} (B_k) (1-\epsilon_k 
            )^p.
        \end{equation*}
        Hence, it suffices to show
        \begin{equation*}
            \lim\limits_{N\to \infty}\mathbb{P} (B_k) =1.
        \end{equation*}
        To get an estimate for this, we consider the following equivalent way to construct $L_i^{(k)}:$ one starts with $L_i \leftarrow \{i\}$ and repeat the following $k$ times:
        \begin{enumerate}
            \item Set $L_{tmp} \leftarrow L_i$ and $A=\emptyset.$
            \item Loop the following until $L_{tmp}$ is empty.
            \begin{enumerate}
                \item [(a)] Pick a particle $i_1 \in L_{tmp},$ then choose $p-1$ particles from $\{1,\cdots,N\}\backslash A \cup \{i_1\}$ denoted by $\{ i_2,\cdots , i_p\}.$
                \item [(b)] Set $L_i\leftarrow L_i \cup \{ i_2, \cdots, i_p\}.$
                \item [(c)] Set $A\leftarrow A \cup \{ i_1, i_2, \cdots, i_p\}.$
                \item [(d)] Set $L_{tmp} \leftarrow L_{tmp} \backslash \{ i_1, i_2, \cdots, i_p\}.$
            \end{enumerate}
        \end{enumerate}
        In the above, we are actually looking back from $t_{k-1}.$ In the $j$-th iteration, we are constructing batches at $t_{k-j}.$ Hence, this is an equivalent way to construct $L_i^{(k)}.$

        Now, we estimate $\mathbb{P}(B_k)$ by constructing the lists $L_{j_\ell}^{(k)}: 1\le \ell\le p$ for $j_\ell \in \mathcal{C}_q^{(k)}$ using the above procedure. Consider that the lists for $j_1,\cdots,j_{\ell -1}$ have been constructed, which have included at most $(\ell -1)p^k$ particles. Now, for $L_{j_\ell}^{(k)}$ not to intersect with the previous lists, one has to choose particles from $\{1,\cdots,N\}\backslash [\cup_{z=1}^{\ell -1} L_{j_z}^{(k)}\cup A \cup \{i_1\}]$ in 2.(a) step. Conditioning on the specific choices of $L_{j_\ell}^{(k)}:$ $1\le\ell \le p$ and $A$ with 
        \begin{equation*}
            N_1:= |L_{j_\ell}^{(k)}\cup A\cup \{i_1\}|, \, N_2:= |A|,
        \end{equation*}
        this probability is controlled from below by
        \begin{equation*}
            \frac{\tbinom{N- N_1}{p-1}}{\tbinom{N-1- N_2}{p-1}}
            \ge \frac{\tbinom{N- 1-\ell p^k}{p-1}}{\tbinom{N-1}{p-1}}.
        \end{equation*}
        Hence, as $N\to \infty,$
        \begin{equation*}
            \mathbb{P}(B_k)\ge \Pi_{\ell=1}^p \left[\frac{\tbinom{N- 1-\ell p^k}{p-1}}{\tbinom{N-1}{p-1}}\right]^k = 1- \mathcal{O}(N^{-1}).
        \end{equation*}
        Hence, $\lim\limits_{N\to \infty}\epsilon_{k+1}=0$ and the claim follows.
    \end{proof}

\subsection{Proof of Theorem \ref{thm: 1}}\label{subsec: pfthm1}

    Now we can prove Theorem \ref{thm: 1}.
    Recall that the TV distance is defined by
    \begin{equation*}
         |\mu - \nu|_{TV} := \sup_{A\in \mathcal{F}} |\mu(A) - \nu(A)|,
    \end{equation*}
    for two probability measures $\mu$, $\nu$ defined on certain sample space $\Omega$ with events $\mathcal{F}.$

    We denote the diameters of the compact support in the spatial and velocity variables at time $t$ by $\mathcal{D}_X(t)$ and $\mathcal{D}_V(t)$, respectively, i.e.
    \begin{equation}\label{eq:3.13}
        \mathcal{D}_X(t) := \max\limits_{i,j}| X_i-X_j|, \quad \mathcal{D}_V(t) := \max\limits_{i,j}| V_i-V_j|.
    \end{equation}
    For the RBM system \eqref{eq: the RBMcs}, Ha et al. \cite{ha2021uniform} gave an estimate on $\mathcal{D}_{X^R}$ and $\mathcal{D}_{V^R}$ shown in Lemma \ref{ineq: coarseD}. This estimate works on any fixed sequence of batches and we provide the proof in Appendix \ref{ass:A} for completeness.

        \begin{lemma}[\cite{ha2021uniform}, Lemma 2.4]\label{ineq: coarseD}
        Let $X^R,V^R$ be a solution of the RBM model \eqref{eq: the RBMcs} with initial data $f_0$ satisfying 
        $$
        \mathcal{D}_X(0) + \mathcal{D}_V(0)<\infty.
        $$
        Then one has
        $$
        \mathcal{D}_{V^R} (t) \le \mathcal{D}_V(0), \quad \mathcal{D}_{X^{R}} (t)\le\mathcal{D}_X(0) + \mathcal{D}_V(0)t, \quad \forall t \ge 0.
        $$
    \end{lemma}


    \begin{proof}[Proof of Theorem \ref{thm: 1}]
        Denote $z:= (x,v),$ $Z^R_i:=(X_i^R,V_i^R),$ and $|z|^q:=|x|^q + |v|^q.$
        By Lemma \ref{lem: TV}, taking $\mu = \mathcal{Q}_\infty^k (f_0)$ and $\nu = \mathcal{Q}_N^{(k)}(f_0),$ one has
        \begin{equation*}
            W_q (\mathcal{Q}_\infty^k (f_0),\mathcal{Q}_N^{(k)}(f_0)) \le 2^{1-\frac{1}{q}} \left(\int_{\mathbb{R}^{2d}} |z|^q d | \mathcal{Q}_\infty^k (f_0)-\mathcal{Q}_N^{(k)}(f_0)|(z)\right)^{\frac{1}{q}}.
        \end{equation*}

       Note that Lemma \ref{ineq: coarseD} degenerates into the original Cucker-Smale model when $p=N.$ Hence for $t\in[0,t_1),$ any $\{Z_i^{(0)}(t)\}_{i=1}^p$ satisfying \eqref{eq: the RBM limit discrete}, one obtains
       \begin{equation*}\label{eq:DVt1}
           \mathcal{D}_{V^{(0)}}(t)\le\mathcal{D}_{V}(0), \quad \mathcal{D}_{X^{(0)}} (t) \le \mathcal{D}_{X} (0) + \mathcal{D}_{V}(0)t.
       \end{equation*}
       Also, by the assumption of $f_0,$ one has $\Bar{v}=0,$ and therefore $\Bar{x}=0.$ Since \eqref{eq: the RBM limit discrete} holds independent of the choice $Z_i,$ defining 
       \begin{equation*}
       \begin{aligned}
           \mathsf{X}(t) :=& \max \{ |x|  \mid \exists \,v,\, s.t. \,f(x,v;t) >0\},\\
           \mathsf{V}(t) :=& \max \{ |v|  \mid \exists \,x,\, s.t.\, f(x,v;t) >0\},
       \end{aligned}
       \end{equation*}
       one has for $t\in[0,t_1],$
       \begin{equation}\label{eq:sfXV}
           \mathsf{V}(t)\le C, \quad \mathsf{X}(t)\le C(1+t),
       \end{equation}
       where $C$ is a constant depending on $R.$ 
       Then for $t\in[t_1,t_2],$ for the new subsystem $\{Z_i^{(1)}\}_{i=1}^p$ with initial distribution $f_{t_1},$ one can use Lemma \ref{ineq: coarseD} again to get
       \begin{equation*}
       \begin{aligned}
           \mathcal{D}_{V^{(1)}}(t)&\le\mathcal{D}_{V^{(1)}}(0)\le 2\mathsf{V}(t_1) \le C, \\
           \mathcal{D}_{X^{(1)}} (t) &\le \mathcal{D}_{X^{(1)}} (0) + \mathcal{D}_{V^{(1)}}(t-\tau)
           \le 2\mathsf{X}(t_1) + 2\mathsf{V}(t_1)(t-\tau)\le C(1+t),
       \end{aligned}
       \end{equation*}
       and \eqref{eq:sfXV} holds for $t\in[t_1,t_2].$
       By induction, for any $t,$ \eqref{eq:sfXV} holds. Therefore $|\mathsf{X}(t_k)|^q + |\mathsf{V}(t_k)|^q \le C(1+t_k)^q.$       
       
       Also, by Lemma \ref{ineq: coarseD}, it holds $|Z^R_i|^q\le C(1+t_k)^q.$
       Then for any $z\in \text{supp}[| \mathcal{Q}_\infty^k (f_0)-\mathcal{Q}_N^{(k)}(f_0)|],$ $|z|^q\le C(1+t_k)^q.$ Therefore, 
       \begin{equation*}
           \int_{\mathbb{R}^{2d}} |z|^q d | \mathcal{Q}_\infty^k (f_0)-\mathcal{Q}_N^{(k)}(f_0)|(z) 
           \le  C(1+t_k)^q | \mathcal{Q}_\infty^k (f_0)-\mathcal{Q}_N^{(k)}(f_0)|_{TV}.
        \end{equation*}
       
       By Lemma \ref{lem: 2ek}, one has 
       \begin{equation*}
           \int_{\mathbb{R}^{2d}} |z|^q d | \mathcal{Q}_\infty^k (f_0)-\mathcal{Q}_N^{(k)}(f_0)|(z) \le C(1+t_k)^q \epsilon_k.
       \end{equation*}
       Hence $W_q (\mathcal{Q}_\infty^k (f_0),\mathcal{Q}_N^{(k)}(f_0)) \le C(1+t_k)\epsilon_k^{\frac{1}{q}}.$ Then by Lemma \ref{lem: thm3.2}, we finish the proof.
    \end{proof}

    \begin{remark}
        The above analysis is not enough to get the mean-field limit independent of $\tau$ and for fixed $N,$ $\epsilon_k\to 1$ as $k\to\infty.$ This is a limitation of the technique in the paper.
    \end{remark}

\section{Proof of Theorem \ref{thm: 2}}\label{sec: mainpf2}

    Here gives an outline of the proof of Theorem \ref{thm: 2}. First, we show the stability of $\mathcal{Q}_\infty$ in Section \ref{subsec: stability}. Second, an auxiliary system is introduced in Section \ref{subsec: auxiliary sys} as a bridge to estimate the $W_q$ distance between $\mathcal{Q}_\infty(\Tilde{f}_{t_k})$ and $\Tilde{f}_{t_{k+1}}$. 
    Last,  the proof of Theorem \ref{thm: 2} is given in Section \ref{subsec:pfthm2}.

\subsection{Stability of \texorpdfstring{$\mathcal{Q}_\infty$}{}}\label{subsec: stability}

In Lemma \ref{lem: lmmG} and Lemma \ref{lem:G stability}, we give the stability analysis on the operator $\mathcal{Q}_\infty.$
\begin{lemma} \label{lem: lmmG}
    Suppose two probability measures $\mu_1,\mu_2$ satisfying $\int vd\mu_1=\int vd\mu_2,$ with compact supports supp$[\mu_1] ,$ supp$[\mu_2]$ respectively, where supp$[\mu_i]\subset\mathbb{R}^d\times B(0,R_\mu)$, $i=1,2,$ and the constant $R_\mu >0.$ Then it holds
    \begin{equation} \label{eq: lmmG}
        W_q (\mathcal{Q}_\infty (\mu_1) ,\mathcal{Q}_\infty (\mu_2)) \le CW_q(\mu_1,\mu_2), \quad q=[1,\infty),
    \end{equation}
    where $C(\kappa,\psi,R_\mu)$ is a positive constant.
\end{lemma}

\begin{proof}
    Set $(X_i,V_i)(0)\sim \mu_1,$ and $(\Bar{X}_i,\Bar{V}_i)(0)\sim \mu_2,$ for $i=1, \cdots, p.$ The key to the proof is to make use of Lemma \ref{lem: SDDI} to get the contraction property of the velocity difference. 
    The two steps are as follows.
    \paragraph{Step 1.}
    First, we show an auxiliary estimate on $\mathcal{D}_V(t).$ By following the formulation of [\cite{ha2018uniform}, Lemma 2.2], denote
    \begin{equation*}
        \Psi_{ij}:= \frac{\psi(|X_i-X_j|)}{p} + \left( 1- \frac{\sum\limits_{j=1}^p \psi(|X_i-X_j|)}{p}\right)\delta_{ij},
    \end{equation*}
    where $\delta_{ij}$ is the Kronecker delta. Then $\Psi_{ij}$ is nonnegative and satisfies
    \begin{equation*}
        \Psi_{ij} \ge \frac{\psi(|X_i -X_j |)}{p}, \quad
        \sum\limits_{j=1}^p \Psi_{ij} =1,\quad
        \sum\limits_{j=1}^p\Psi_{ij}(V_j-V_i) = \sum\limits_{j=1}^{p}\frac{\psi(|X_i -X_j |)}{p}(V_j-V_i).
    \end{equation*}
    We denote $V_i = (V_i^k)$ and obtain the following equality
    \begin{align}\label{eq:Dv9}
        \frac{1}{2}\frac{d}{dt}&|V_i - V_j|^2
        = \frac{1}{2}\frac{d}{dt}\sum\limits_{k=1}^d|V_i^k - V_j^k|^2\notag\\
        =& \sum\limits_{k=1}^d(V_i^k - V_j^k)(\Dot{V}_i^k - \Dot{V}_j^k)\notag\\
        =& \sum\limits_{k=1}^d \frac{\kappa p}{p-1} (V_i^k - V_j^k)\left[
        \sum\limits_{m=1}^p \Psi_{im}(V_m^k - V_i^k) - \sum\limits_{m=1}^p \Psi_{jm}(V_m^k - V_j^k)\right]\notag\\
        =& -\sum\limits_{k=1}^d \frac{\kappa p}{p-1} (V_i^k - V_j^k)^2
        +  \sum\limits_{k=1}^d \kappa (V_i^k - V_j^k)\left( \sum\limits_{m=1}^p (\Psi_{im}-\Psi_{jm})V_m^k\right)\notag\\
        =& -\sum\limits_{k=1}^d \frac{\kappa p}{p-1} (V_i^k - V_j^k)^2\notag\\
        &+  \sum\limits_{k=1}^d \frac{\kappa p}{p-1} (V_i^k - V_j^k)\left[ \sum\limits_{m=1}^p \left(\Psi_{im}-\min\{\Psi_{im},\Psi_{jm}\} + \min\{\Psi_{im},\Psi_{jm}\} - \Psi_{jm}\right)V_m^k\right].
    \end{align}
    For a given time $t,$ one can choose indices $i_t$ and $j_t$ satisfying the relation 
    \begin{equation*}
        \mathcal{D}_V(t) = |V_{i_t}(t) -V_{j_t}(t)|,
    \end{equation*}
    and thus for each $m=1,2,\cdots,p,$
    \begin{equation}\label{eq:Dv10}
        \sum\limits_{k=1}^d (V_{i_t}^k -V_{j_t}^k)V_{j_t}^k \le  \sum\limits_{k=1}^d (V_{i_t}^k -V_{j_t}^k)V_{m}^k \le  \sum\limits_{k=1}^d (V_{i_t}^k -V_{j_t}^k)V_{i_t}^k.
    \end{equation}
    Combining \eqref{eq:Dv9} and \eqref{eq:Dv10}, one obtains
    \begin{align}\label{eq:Dv11}
        \frac{1}{2}\frac{d}{dt}|V_{i_t} - V_{j_t}|^2       
        \le&  -\sum\limits_{k=1}^d \frac{\kappa p}{p-1} (V_{i_t}^k - V_{j_t}^k)^2\\
        &+  \sum\limits_{k=1}^d \frac{\kappa p}{p-1} (V_{i_t}^k - V_{j_t}^k)V_{i_t}^k \sum\limits_{m=1}^p \left(\Psi_{im}-\min\{\Psi_{im},\Psi_{jm}\}\right) \notag\\
        &-  \sum\limits_{k=1}^d \frac{\kappa p}{p-1} (V_{i_t}^k - V_{j_t}^k)V_{j_t}^k \sum\limits_{m=1}^p \left(\Psi_{jm} -\min\{\Psi_{im},\Psi_{jm}\}\right)\notag\\
        =&-\sum\limits_{m=1}^p\frac{\kappa p}{p-1} \min\{\Psi_{im},\Psi_{jm}\}|V_{i_t} - V_{j_t}|^2\notag\\
        \le & - \frac{\kappa p}{p-1}\psi_0 |V_{i_t} - V_{j_t}|^2.
    \end{align}
    Note that $\mathcal{D}_V$ is Lipschitz continuous with respect to $t$ and thus it is differentiable almost everywhere. Consider the instant $t$ at which $\mathcal{D}_V$ is differentiable. Then,
    \begin{equation}\label{eq:Dv12}
        \frac{d}{dt}\mathcal{D}_V^2 \le \frac{d}{dt}|V_{i_t} - V_{j_t}|^2.
    \end{equation}
    Combining with \eqref{eq:Dv11} and \eqref{eq:Dv12}, one obtains
    \begin{equation*}
        \left|\frac{d}{dt}\mathcal{D}_X\right| \le \mathcal{D}_V, \quad 
        \frac{d}{dt}\mathcal{D}_V \le -\frac{\kappa p}{p-1}\psi_0 \mathcal{D}_V,
    \end{equation*}
    which is an SDDI. By Lemma \ref{lem: SDDI}, there exists a positive constant $R_\mu$ such that 
    \begin{equation}\label{eq:Dv}
        \mathcal{D}_V(t)\le 2e^{-\kappa\psi_0 t}R_\mu.
    \end{equation}

    \paragraph{Step 2.}
    
    It can be obtained that 
    \begin{equation}\label{eq:3.29}
        \frac{d}{dt} \mathbb{E} | V_1 - \Bar{V}_1 |^q =  
           \frac{1}{p}\frac{d}{dt}\mathbb{E}\sum\limits_{i=1}^p \left(\sum\limits_{k=1}^d | V_i^k - \Bar{V}_i^k |^2\right)^{\frac{q}{2}},
    \end{equation}
    with
    \begin{equation*}
    \begin{aligned}
        \frac{d}{dt}\mathbb{E} &\left(\sum\limits_{k=1}^d | V_i^k - \Bar{V}_i^k |^2\right)^{\frac{q}{2}} 
        =  \mathbb{E}\left[ \frac{q}{2} \left(\sum\limits_{k=1}^d | V_i^k - \Bar{V}_i^k |^2\right)^{\frac{q}{2}-1} \frac{d}{dt} \sum\limits_{k=1}^d 
|V_i^k - \Bar{V}_i^k|^2 \right]\\
        = &\frac{q\kappa }{p-1} \mathbb{E} \left[ |V_i - \Bar{V}_i|^{q-2} \sum\limits_{j=1}^p \psi(|X_j - X_i|)\sum\limits_{k=1}^d(V_j^k - V_i^k - \Bar{V}_j^k + \Bar{V}_i^k)( V_i^k - \Bar{V}_i^k)\right. \\
        &\left.  + \sum\limits_{j=1}^p 
        |V_i - \Bar{V}_i|^{q-2}\left(\psi(|X_j - X_i|) - \psi(|\Bar{X}_j - \Bar{X}_i|)\right)
       \sum\limits_{k=1}^d(\Bar{V}_j^k - \Bar{V}_i^k)( V_i^k - \Bar{V}_i^k) 
       \right].
    \end{aligned}
    \end{equation*}
    Rewrite the right-hand side of $\eqref{eq:3.29}$ as
   \begin{equation*}
      \frac{d}{dt} \mathbb{E} | V_1 - \Bar{V}_1 |^q
       := I_{21} + I_{22},
   \end{equation*}
   where
   \begin{equation*}
       I_{21}:= \frac{q\kappa}{p-1}\mathbb{E}  \sum\limits_{i,j=1}^p \sum\limits_{k=1}^d  \psi(|X_j - X_i|) |V_i - \Bar{V}_i|^{q-2} (V_j^k - V_i^k - \Bar{V}_j^k + \Bar{V}_i^k)(V_i^k - \Bar{V}_i^k),
   \end{equation*}
   and 
   \begin{multline*}
       I_{22}:= \frac{q\kappa}{p-1}\mathbb{E}  \sum\limits_{i,j=1}^p \sum\limits_{k=1}^d
        \left(\psi(|X_j - X_i|) - \psi(|\Bar{X}_j - \Bar{X}_i|)\right)\\
        \cdot|V_i - \Bar{V}_i|^{q-2}
       ( V_i^k - \Bar{V}_i^k )(\Bar{V}_j^k - \Bar{V}_i^k). 
   \end{multline*}
   First estimate $I_{21}$ using the index changing trick:
   \begin{equation*}
       \begin{aligned}
           I_{21} = &\frac{q\kappa}{p-1}\mathbb{E}  \sum\limits_{i,j=1}^p \sum\limits_{k=1}^d  \psi(|X_j - X_i|) |V_i - \Bar{V}_i|^{q-2} [(V_j^k -\Bar{V}_j^k)-( V_i^k - \Bar{V}_i^k )](V_i^k - \Bar{V}_i^k)\\
           = & \frac{q\kappa}{p-1}\mathbb{E} \left[ \sum\limits_{i,j=1}^p \sum\limits_{k=1}^d  \psi(|X_j - X_i|) |V_i - \Bar{V}_i|^{q-2} (V_j^k -\Bar{V}_j^k)(V_i^k - \Bar{V}_i^k)\right.\\
           & \left.- \sum\limits_{i,j=1}^p \sum\limits_{k=1}^d  \psi(|X_j - X_i|) |V_i - \Bar{V}_i|^{q-2}(V_i^k - \Bar{V}_i^k)^2\right]\\
           = & \frac{q\kappa}{p-1}\mathbb{E} \left[ \sum\limits_{i,j=1}^p \sum\limits_{k=1}^d  \psi(|X_j - X_i|) |V_j - \Bar{V}_j|^{q-2} (V_j^k -\Bar{V}_j^k)(V_i^k - \Bar{V}_i^k)\right.\\
           & \left.- \sum\limits_{i,j=1}^p \sum\limits_{k=1}^d  \psi(|X_j - X_i|) |V_i - \Bar{V}_i|^{q-2}(V_i^k - \Bar{V}_i^k)^2\right]\\
           = & -\frac{q\kappa}{p-1}\mathbb{E} \left[ \sum\limits_{i,j=1}^p \sum\limits_{k=1}^d  \psi(|X_j - X_i|)\right. \\
           &\cdot \left[|V_i - \Bar{V}_i|^{q-2} (V_i^k -\Bar{V}_i^k) - |V_j - \Bar{V}_j|^{q-2} (V_j^k -\Bar{V}_j^k)\right]
           \Biggl.(V_i^k - \Bar{V}_i^k)\Biggr]\\
           = & \frac{q\kappa}{2(p-1)}\mathbb{E} \left[ \sum\limits_{i,j=1}^p \sum\limits_{k=1}^d  \psi(|X_j - X_i|) \right.\\
           & \Biggl.\cdot \left[|V_i - \Bar{V}_i|^{q-2} (V_i^k -\Bar{V}_i^k) - |V_j - \Bar{V}_j|^{q-2} (V_j^k -\Bar{V}_j^k)\right][(V_j^k - \Bar{V}_j^k)-(V_i^k - \Bar{V}_i^k)]\Biggr].
       \end{aligned}
   \end{equation*}
    One can use the relation 
    \begin{equation*}
        \sum\limits_{k=1}^d\left[|V_i - \Bar{V}_i|^{q-2} (V_i^k -\Bar{V}_i^k) - |V_j - \Bar{V}_j|^{q-2} (V_j^k -\Bar{V}_j^k)\right]
        \cdot \left.[(V_j^k - \Bar{V}_j^k)-(V_i^k - \Bar{V}_i^k)]\right] \le 0,
    \end{equation*}
    and the condition \eqref{ass:psi} of $\psi$ to obtain
    \begin{equation*}
        \begin{aligned}
        I_{21} \le & \frac{q\kappa}{2(p-1)}\psi_0 \mathbb{E}\Bigg\{\sum\limits_{i,j=1}^p \sum\limits_{k=1}^d 
        \left[|V_i - \Bar{V}_i|^{q-2} (V_i^k -\Bar{V}_i^k) - |V_j - \Bar{V}_j|^{q-2} (V_j^k -\Bar{V}_j^k)\right]\\
        &\cdot [(V_j^k - \Bar{V}_j^k)-(V_i^k - \Bar{V}_i^k)]\Bigg\}\\
        = &\frac{q\kappa}{2(p-1)}\psi_0 \mathbb{E}\left[\sum\limits_{i=1}^p \sum\limits_{k=1}^d 
        |V_i - \Bar{V}_i|^{q-2} (V_i^k -\Bar{V}_i^k)\sum\limits_{j=1}^p (V_j^k -\Bar{V}_j^k)\right.\\
        &-  \sum\limits_{i=1}^p \sum\limits_{k=1}^d 
         p|V_i - \Bar{V}_i|^{q-2} (V_i^k -\Bar{V}_i^k)^2
        -  \sum\limits_{j=1}^p \sum\limits_{k=1}^d 
         p|V_j - \Bar{V}_j|^{q-2} (V_j^k -\Bar{V}_j^k)^2\\
        &+  \left.\sum\limits_{j=1}^p \sum\limits_{k=1}^d 
         |V_j - \Bar{V}_j|^{q-2} (V_j^k -\Bar{V}_j^k)\sum\limits_{i=1}^p (V_i^k -\Bar{V}_i^k)\right].
        \end{aligned}
    \end{equation*}
    By the momentum condition $\mathbb{E}\sum\limits_{i=1}^p (V_i^k - \Bar{V}_i^k) = 0,$ it holds
    \begin{equation}\label{eq:3.38}
        I_{21} \le  -\frac{p^2q\kappa}{p-1}\psi_0 \mathbb{E}|V_1 - \Bar{V}_1|^q.
    \end{equation}

   For $I_{22},$ one has
    \begin{equation*}
        \begin{aligned}
            I_{22} \le & \frac{q\kappa}{p-1}\|\psi\|_\text{{Lip}} \mathbb{E} \left[
            \sum\limits_{i,j=1}^p \sum\limits_{k=1}^d \left| |X_j - X_i| - 
            |\Bar{X}_j - \Bar{X}_i | \right| V_i - \Bar{V}_i|^{q-2}|V_i^k - \Bar{V}_i^k| |\Bar{V}_j^k - \Bar{V}_i^k|
            \right]\\
            \le & \frac{q\kappa}{p-1}\|\psi\|_\text{{Lip}} \mathbb{E} \left[
            \sum\limits_{i,j=1}^p\left| |X_j - X_i| - 
            |\Bar{X}_j - \Bar{X}_i | \right| |V_i - \Bar{V}_i|^{q-1} |\Bar{V}_j - \Bar{V}_i|
            \right],
        \end{aligned}
    \end{equation*}
    by the Cauchy inequality. Then by the inequality $\left| |a-b| - |c-d| \right| \le |a-c| + |b-d|,$ one has
   \begin{equation}\label{eq:C(d)2}
        \begin{aligned}
            I_{22} \le  & \frac{q\kappa}{p-1}\|\psi\|_\text{{Lip}} \mathbb{E} \left[
            \sum\limits_{i,j=1}^p\left| |X_i - \Bar{X}_i| + 
            |X_j - \Bar{X}_j | \right| |V_i - \Bar{V}_i|^{q-1} |\Bar{V}_j - \Bar{V}_i|
            \right]\\
           \le  & \frac{q\kappa}{p-1}\|\psi\|_\text{{Lip}}  \mathcal{D}_{\Bar{V}}(t) \mathbb{E} \left[p
            \sum\limits_{i=1}^p |X_i - \Bar{X}_i| | V_i - \Bar{V}_i |^{q-1} + 
            \sum\limits_{i,j=1}^p |X_j - \Bar{X}_j| | V_i - \Bar{V}_i |^{q-1}\right]\\
            \le & \frac{2q p^2\kappa}{p-1}\|\psi\|_\text{{Lip}}   
            \mathcal{D}_{\Bar{V}}(t) \mathbb{E} \left[
             |X_1 - \Bar{X}_1| | V_1 - \Bar{V}_1 |^{q-1} \right],
        \end{aligned}
    \end{equation}
    by H\"older's inequality. Now by combining \eqref{eq:3.29} and \eqref{eq:3.38} and \eqref{eq:C(d)2}, one has
    \begin{equation*}
        \begin{aligned}
            \frac{d}{dt}\mathbb{E}| V_1 - \Bar{V}_1 |^q \le & -\frac{p}{p-1}q\kappa\psi_0 
            \mathbb{E}|V_1 - \Bar{V}_1|^q \\
            &+ \frac{2p}{p-1}q\kappa \|\psi\|_\text{{Lip}}\mathcal{D}_{\Bar{V}}(t)
            \mathbb{E}[|X_1-\Bar{X}_1| |V_1 - \Bar{V}_1|^{q-1}]\\
            \le & -\frac{p}{p-1}q\kappa\psi_0 
            \mathbb{E}|V_1 - \Bar{V}_1|^q \\
            &+ \frac{2p}{p-1}q\kappa \|\psi\|_\text{{Lip}}\mathcal{D}_{\Bar{V}}(t)
            (\mathbb{E}[|X_1-\Bar{X}_1|^q])^{\frac{1}{q}} (\mathbb{E}|V_1 - \Bar{V}_1|^q)^{1-\frac{1}{q}},
        \end{aligned}
    \end{equation*}
    similarly by H\"older's inequality. Since
    \begin{equation}\label{eq:qqv}
        \begin{aligned}
            \frac{d}{dt} (\mathbb{E}&|V_1 - \Bar{V}_1|^q)^{\frac{1}{q}} = 
            \frac{1}{q}(\mathbb{E}|V_1 - \Bar{V}_1|^q)^{\frac{1}{q}-1} \frac{d}{dt}\mathbb{E}|V_1 - \Bar{V}_1|^q\\
            \le & -\frac{p\kappa}{p-1}\psi_0 
            (\mathbb{E}[|V_1 - \Bar{V}_1|^q_q])^{\frac{1}{q}} 
            + \frac{2p\kappa}{p-1} \|\psi\|_\text{{Lip}}\mathcal{D}_{\Bar{V}}(t)
            (\mathbb{E}[|X_1-\Bar{X}_1|^q])^{\frac{1}{q}}.
        \end{aligned}
    \end{equation}
    and 
    \begin{equation*}
        \begin{aligned}
            \frac{d}{dt} \mathbb{E}|X_1 - \Bar{X}_1|^q = &
            \mathbb{E}\frac{q}{2}\left(\sum\limits_{k=1}^d |X_1^k - \Bar{X}_1^k|^2\right)^{\frac{q}{2}-1} \sum\limits_{k=1}^d(X_1^k - \Bar{X}_1^k)\cdot 2(V_1^k - \Bar{V}_1^k)\\
            \le & \mathbb{E} q |X_1 - X_1|^{q-1} |V_1 - \Bar{V}_1 |\\
            \le & q (\mathbb{E}|X_1 - \Bar{X}_1|^q)^{1-\frac{1}{q}} (\mathbb{E}|V_1 - \Bar{V}_1|^q)^{\frac{1}{q}},
        \end{aligned}
    \end{equation*}
    thus it holds 
    \begin{equation}\label{eq:qqx}
        \frac{d}{dt}(\mathbb{E}|X_1 - \Bar{X}_1|^q)^{\frac{1}{q}} = 
        \frac{1}{q}(\mathbb{E}|X_1 - \Bar{X}_1|^q)^{\frac{1}{q}-1}
        \frac{d}{dt} \mathbb{E}|X_1 - \Bar{X}_1|^q
        \le (\mathbb{E}|V_1 - \Bar{V}_1|^q)^{\frac{1}{q}}.
    \end{equation}
    Note that $\mathcal{D}_{\Bar{V}}\le 2 e^{-\kappa\psi_0 t}  R_\mu$ by Step 1. 
    Now one knows that both $(\mathbb{E}|X_1 - \Bar{X}_1|^q)^{\frac{1}{q}}$ and $(\mathbb{E}|V_1 - \Bar{V}_1|^q)^{\frac{1}{q}}$ satisfy Lemma \ref{lem: SDDI} and then there exists a constant $C$ such that
    \begin{equation*}
        (\mathbb{E}|X_1 - \Bar{X}_1|^q)^{\frac{1}{q}}(t)
        + (\mathbb{E}|V_1 - \Bar{V}_1 \|^q)^{\frac{1}{q}}(t)
        \le C \left[(\mathbb{E}|X_1 - \Bar{X}_1 |^q)^{\frac{1}{q}}(0)
        + (\mathbb{E}|V_1 - \Bar{V}_1 |^q)^{\frac{1}{q}}(0)\right].
    \end{equation*}
    Then, 
    \begin{equation*}
        \begin{aligned}
         W_q (\mathcal{Q}_\infty (\mu_1) ,\mathcal{Q}_\infty (\mu_2)) \le & \left(\mathbb{E} |X_1 -\Bar{X}_1|^q (t) + \mathbb{E} |V_1 -\Bar{V}_1|^q(t)\right)^{\frac{1}{q}}\\
         \le & \left(\mathbb{E}|X_1 - \Bar{X}_1|^q (t)\right)^{\frac{1}{q}}
         + \left(\mathbb{E}|V_1 - \Bar{V}_1|^q (t)\right)^{\frac{1}{q}}\\
         \le & C \left[(\mathbb{E}|X_1 - \Bar{X}_1|^q)^{\frac{1}{q}}(0)
        + (\mathbb{E}|V_1 - \Bar{V}_1|^q)^{\frac{1}{q}}(0)\right]\\
         \le &2 C \left[\mathbb{E} |X_1 -\Bar{X}_1|^q (0) + \mathbb{E} |V_1 -\Bar{V}_1|^q(0)\right]^{\frac{1}{q}},
    \end{aligned} 
    \end{equation*}
    since $a^{\frac{1}{q}}+ b^{\frac{1}{q}}\le 2^{1-\frac{1}{q}}(a+b)^\frac{1}{q}$ for $a\ge0,$ $b\ge 0$ with $q\ge 1.$ For any $\epsilon>0,$ the initial data can be chosen by a coupling such that 
    \begin{equation*}
        \left[\mathbb{E} |X_1 -\Bar{X}_1|^q (0) + \mathbb{E} |V_1 -\Bar{V}_1|^q(0)\right]^{\frac{1}{q}} \le \epsilon + W_1(\mu_1,\mu_2).
    \end{equation*}
    Hence by the arbitrariness of $\epsilon,$ it holds $
    W_q (\mathcal{Q}_\infty (\mu_1) ,\mathcal{Q}_\infty (\mu_2)) \le C W_q(\mu_1,\mu_2),$
\end{proof}

    \begin{lemma}\label{lem:G stability}
    For $\Tilde{f}$ defined in $\eqref{eq: mfl for cs}$ with $R<\infty$, one has
    \begin{equation*}
        W_q(\mathcal{Q}_{\infty}^{n-m+1}(\Tilde{f}_{m-1}),\mathcal{Q}_{\infty}^{n-m}(\Tilde{f}_{m}))
        \le C W_q(\mathcal{Q}_\infty(\Tilde{f}_{m-1}), \Tilde{f}_m), \quad q=[1,\infty),
    \end{equation*}
    where $C$ is relevant to $\psi,\kappa,R $ and $ d.$
    \end{lemma}
    \begin{proof}
        By the flocking property shown in Step 1 in Lemma \ref{lem: lmmG}, one has $\text{supp}^v[\mathcal{Q}_{\infty}(\Tilde{f}_{m-1})]\subset \text{supp}^v[\Tilde{f}_{m-1}].$ Also, $\text{supp}^v[\Tilde{f}_{m-1}]\subset \text{supp}^v[\Tilde{f}_0]$ by Lemma \ref{lem: supp}. Furthermore, for $n-m-1\ge k\ge 0,$ $\mathcal{Q}^{k+1}_\infty (\Tilde{f}_{m-1})$ and $\mathcal{Q}^k_\infty (\Tilde{f}_m)$ have compact supports, where $\mathcal{Q}^0_\infty (\Tilde{f}_m):=\Tilde{f}_m.$
        
        Similarly with the proof of Proposition \ref{prop: M1v}, note that $(X_1(k\tau),V_1(k\tau)$ and $(\Bar{X}_1(k\tau), \Bar{V}_1(k\tau))$ obey $\mathcal{Q}^{k+1}_\infty (\Tilde{f}_{m-1})$ and $\mathcal{Q}^k_\infty (\Tilde{f}_m)$ respectively, which satisfy the momentum equality assumption of Lemma \ref{lem: lmmG}.

        Denote 
        \[
        |\Delta X_1|:=(\mathbb{E}|X_1 - \Bar{X}_1|^q)^{\frac{1}{q}},\quad |\Delta V_1|:=(\mathbb{E}|V_1 - \Bar{V}_1|^q)^{\frac{1}{q}}, \quad
        |\Delta Z_1|:=|\Delta X_1|+|\Delta V_1|,
        \]
        for $q\ge 1$.
        Hence, by the analysis in Lemma \ref{lem: lmmG}, \eqref{eq:qqx} and \eqref{eq:qqv} still hold for $t\in [0,(n-m)\tau].$  
        By Lemma \ref{lem:G stability}, one obtains
        \begin{equation*}
            |\Delta Z_1 ((n-m)\tau)| \le C |\Delta Z_1 (0)|.
        \end{equation*}
        By choosing a suitable coupling of the initial data, the proof is finished.
    \end{proof}

    \subsection{An auxiliary system}\label{subsec: auxiliary sys}
    
    We introduce an auxiliary system \eqref{eq: tempsys} with law $\Tilde{S}(\Tilde{f}_{t_k})$ to get the estimate on $$W_q (\mathcal{Q}_\infty (\Tilde{f}_{t_k}), \Tilde{f}_{t_{k+1}}) \le W_q(\Tilde{S}(\Tilde{f}_{t_k}),\Tilde{f}_{t_{k+1}}) + W_q(\Tilde{S}(\Tilde{f}_{t_k}),\mathcal{Q}_\infty (\Tilde{f}_{t_{k}}) ).$$ 

    Rewrite \eqref{eq: mfl for cs} into the following particle system
    \begin{equation} \label{eq: mflsys}
        \left\{\begin{aligned}
            \frac{d}{dt} \Tilde{X}_i(t) =& \Tilde{V}_i(t),\, i=1,\cdots, N,\\
            \frac{d}{dt} \Tilde{V}_i(t) =& \int \kappa \psi(|x-\Tilde{X}_i|)(v-\Tilde{V}_i) \Tilde{f}_t(x,v)dxdv.
        \end{aligned}\right.
    \end{equation}

    In order to show the law of $(\Tilde{X},\Tilde{V})$ at time $n\tau$ is close to $\mathcal{Q}_\infty^n,$ we consider the transient system
    \begin{equation} \label{eq: tempsys}
        \left\{\begin{aligned}
            \frac{d}{dt} \hat{X}_i(t) =& \hat{V}_i(t),\, i=1,\cdots, N,\\
            \frac{d}{dt} \hat{V}_i(t) =& \int \kappa\psi(|x-\hat{X}_i|)(v-\hat{V}_i) \Tilde{f}_{t_k}(x,v)dxdv,
        \end{aligned}\right.
    \end{equation}
    with $(\hat{X}(0),\hat{V}(0))\sim \Tilde{f}_{t_k}.$
    Denote $\Tilde{S}(\Tilde{f}_{t_k}):=\text{Law}(\hat{X}(\tau),\hat{V}(\tau)).$

    \begin{lemma}\label{lem: suppS}
    Under the same assumption of Lemma \ref{lem:G stability}, it holds
        \begin{equation}\label{eq:suppS}
            \text{supp}[\Tilde{S}(\Tilde{f}_{t_k})]\subset
            B(0,C)\times B(0,C e^{-\kappa\psi_0 t_{k+1}}).
        \end{equation}
    \end{lemma}
    \begin{proof}
           For $\Tilde{S}(\Tilde{f}_{t_k}),$ define
        \begin{align*}
            \mathsf{X}_k (t) :=& \max\{ |x(t,x^k,v^k)|  \mid (x^k,v^k)\in \text{supp}[\Tilde{f}_{t_k}]\},\\
            \mathsf{V}_k (t) :=& \max\{ |v(t,x^k,v^k)|  \mid (x^k,v^k)\in \text{supp}[\Tilde{f}_{t_k}]\},
        \end{align*}
        where the pair ($x(t,x^k,v^k), v(t,x^k,v^k)$) denotes $(x(t),v(t))$ being the solution of \eqref{eq: tempsys} with $(x(t_k),v(t_k)) = (x^k,v^k).$ Note that the maximum is reached because it is assumed that $\text{supp}[\Tilde{f}_{t_k}]$ is compact. Denote $K_t^x \subset \text{supp}^x [\Tilde{f}_{t_k}]$ (resp. $K_t^v \subset \text{supp}^v [\Tilde{f}_{t_k}]$) the set of points $(x^k,v^k)$ such that the maximum is reached in $\mathsf{X}_k(t)$ (resp. in $\mathsf{V}_k(t)$). By definition, one has $|\mathsf{X}_k(t)|^2 = |x(t,x^k,v^k)|^2$ for every $(x^k,v^k)\in K_t^x.$ It follows from Danskin's theorem \cite{piccoli2015control,danskin1966theory} and from the fact that $\partial_t x(t,x^k,v^k) = v(t,x^k,v^k)$ that
        \begin{equation*}
            \mathsf{X}_k(t)\cdot \Dot{\mathsf{X}}_k(t) = \max\{ x(t,x^k,v^k) \cdot v(t,x^k,v^k) \mid (x^k,v^k)\in K_t^x \}.
        \end{equation*}
        Using the Cauchy-Schwarz inequality, one obtains that 
        \begin{equation*}
            \Dot{\mathsf{X}}_k(t) \le \mathsf{V}_k(t).
        \end{equation*}
        Similarly, by  Danskin's theorem again, one has
        \begin{multline*}
            \mathsf{V}_k(t)\cdot \Dot{\mathsf{V}}_k(t) = \max\{ v(t,x^k,v^k) \cdot \int \kappa\psi(|x-x(t,x^k,v^k)|)(v-v(t,x^k,v^k)) \Tilde{f}_{t_k}dxdv\\
            \mid (x^k,v^k)\in K_t^v \}.  
        \end{multline*}
        Note that by the definition of $K_t^v,$ $v(t,x^k,v^k) \cdot (v-v(t,x^k,v^k))\le 0.$ It holds that 
        \begin{equation*}
            \mathsf{V}_k(t)\cdot \Dot{\mathsf{V}}_k(t)  \le
            -\kappa \psi_0 \int v(t,x^k,v^k) \cdot v(t,x^k,v^k) \Tilde{f}_{t_k} dxdv,\quad (x^k,v^k)\in K_t^v ,
        \end{equation*}
        since $\int v \Tilde{f}_{t_k} dxdv = 0$ and $\int\Tilde{f}_{t_k}dxdv =1.$ Then one has 
        \begin{equation*}
            \Dot{\mathsf{V}}_k(t) \le -\kappa \psi_0 \mathsf{V}_k(t),
        \end{equation*}
        and thus \eqref{eq:suppS} holds.
    \end{proof}

    \begin{lemma}\label{lem: w2 S f}
    Under the same assumption of Lemma \ref{lem:G stability}, there exists a positive constant $\alpha$ such that
        \begin{equation}\label{eq:3.52}
            W_q(\Tilde{S}(\Tilde{f}_{t_k}),\Tilde{f}_{t_{k+1}})\le C e^{-\kappa\psi_0 t_k} \tau^2, \quad q\in[1,\infty),
        \end{equation}
        with $C=C(\kappa,\psi_0,\|\psi\|_\text{{Lip}}, R).$
    \end{lemma}
    
    \begin{proof}
        Let $\mathcal{R}(x,v,t) = \int \psi(|y-x|)(w-v)(\Tilde{f}_t(y,w)-\Tilde{f}_{t_k}(y,w))dydw.$ For $t\in [t_k,t_{k+1}),$ by straightforward calculus, it holds
        \begin{small}
        \begin{equation*}
            \begin{aligned}
                \left| \frac{\partial \mathcal{R}}{\partial t} (x,v,t)\right| = &
                \left| \int \psi(|y-x|)(w-v) \frac{\partial \Tilde{f}_t}{\partial t}(y,w)dydw\right|\\
                = & \left| \int \psi(|y-x|)(w-v) \Bigl[ - w\cdot \nabla
                _y \Tilde{f}_t(y,w) \Bigr.\right.\\
                &+ \left.\left.\nabla_w \cdot \left(\Tilde{f}_t(y,w)\kappa \int \psi (|\Bar{y}-y|)(\Bar{w}-w)\Tilde{f}_t(\Bar{y},\Bar{w})d\Bar{y}d\Bar{w}\right)\right]dydw \right|\\
                = & \left| \int \biggl[w\cdot \nabla_y \psi(|y-x|)(w-v) \Tilde{f}_t(y,w)\right.\biggr.\\
                & - \left. \left.d\kappa\psi(|y-x|)\Tilde{f}_t(y,w)\int  \psi (|\Bar{y}-y|)(\Bar{w}-w)\Tilde{f}_t(\Bar{y},\Bar{w})d\Bar{y}d\Bar{w}\right]dydw \right|\\
                \le&\|\psi\|_\text{{Lip}} \int |w| |v-w| \Tilde{f}_t (y,w) dydw 
                + d\kappa \int |\Bar{w}-w| \Tilde{f}_t(y,w) \Tilde{f}_t (\Bar{y},\Bar{w}) d\Bar{y}d\Bar{w}dydw\\
                \le&\|\psi\|_\text{{Lip}} \int (|w| |v| + |w|^2) \Tilde{f}_t (y,w) dydw 
                + d\kappa \int 2|w| \Tilde{f}_t(y,w) dydw.
            \end{aligned}
        \end{equation*}
        \end{small}
         By Lemma \ref{lem: supp}, 
         \[ \left|\frac{\partial\mathcal{R}}{\partial t}(\Tilde{X}_1,\Tilde{V}_1,t)\right| \le C e^{-\kappa\psi_0 t}.
         \]
        We first consider $q$ to be an even integer.
        For $(\Tilde{X}_1,\Tilde{V}_1)$ obeying \eqref{eq: mflsys} and $(\hat{X}_1,\hat{V}_1)$ obeying \eqref{eq: tempsys}, with $(\Tilde{X}_1,\Tilde{V}_1) (t_k) = (\hat{X}_1,\hat{V}_1) (t_k) \sim \Tilde{f}_{t_k},$ it holds
        \begin{equation*}
            \begin{aligned}
                \frac{d}{dt} &|\Tilde{V}_1 -\hat{V}_1|^2 =  2\sum\limits_{k=1}^d(\Tilde{V}^k_1 -\hat{V}^k_1)\frac{d}{dt}(\Tilde{V}^k_1 -\hat{V}^k_1)\\
                = &2\kappa \sum\limits_{k=1}^d(\Tilde{V}^k_1 -\hat{V}^k_1)
                 \Bigl[\int \psi (|x-\Tilde{X}_1|) (v - \Tilde{V}_1) \Tilde{f}_{t_k}(x,v) dxdv \Bigr.\\
                &\Bigl.- \int\psi (|x-\hat{X}_1|) (v - \hat{V}_1) \Tilde{f}_{t_k}(x,v) dxdv + \mathcal{R}(\Tilde{X}_1,\Tilde{V}_1)\Bigr]^{(k)}\\
                = & 2\kappa  \sum\limits_{k=1}^d(\Tilde{V}^k_1 -\hat{V}^k_1)\left( \int \psi (|x-\Tilde{X}_1|) [(v - \Tilde{V}_1)-(v-\hat{V}_1)] \Tilde{f}_{t_k}(x,v) dxdv\right)^{(k)}\\
                & + 2\kappa \sum\limits_{k=1}^d(\Tilde{V}^k_1 -\hat{V}^k_1) \left(\int [\psi (|x-\hat{X}_1|)- \psi (|x-\Tilde{X}_1|)](v - \hat{V}_1) \Tilde{f}_{t_k}(x,v) dxdv \right) ^{(k)}\\
                & + 2\kappa  \sum\limits_{k=1}^d(\Tilde{V}^k_1 -\hat{V}^k_1) \mathcal{R}(\Tilde{X}_1,\Tilde{V}_1)^{(k)}\\
                \le &  -q\psi_0 | \Tilde{V}_1 - \hat{V}_1|^2 
                +2\kappa\|\psi\|_{\text{Lip}}   |\Tilde{X}_1-\hat{X}_1 | \sum\limits_{k=1}^d
                (\Tilde{V}^k_1 -\hat{V}^k_1)\left(\int (v-\hat{V}_1)\Tilde{f}_{t_k}\right)^{(k)} \\
                &+ 2\kappa |\Tilde{V}_1 - \hat{V}_1 | |\mathcal{R}(\Tilde{X}_1,\Tilde{V}_1)|,\\
                \le &  -2\kappa\psi_0 | \Tilde{V}_1 - \hat{V}_1|^2 
                +2\kappa\|\psi\|_{\text{Lip}}   |\Tilde{X}_1-\hat{X}_1|  |\Tilde{V}_1 - \hat{V}_1 |
                |\hat{V}_1|\\
                &+ 2\kappa |\Tilde{V}_1 - \hat{V}_1 | |\mathcal{R}(\Tilde{X}_1,\Tilde{V}_1)|,
            \end{aligned}
        \end{equation*}
        where the superscript $(k)$ means the $k$-th dimension of vectors. Since $\int v\Tilde{f}_t dxdv = \Bar{v} =0$ by \eqref{eq:tildevf}, and $|\hat{V}_1|\le C e^{-\kappa\psi_0 t_k}$ by Lemma \ref{lem: supp} and Lemma \ref{lem: suppS}, it holds
        \begin{multline}
            \frac{d}{dt} |\Tilde{V}_1 -\hat{V}_1|^2   
                \le   -2\kappa\psi_0 | \Tilde{V}_1 - \hat{V}_1|^2 
                +C\kappa\|\psi\|_{\text{Lip}}   e^{-\kappa\psi_0 t_k} |\Tilde{X}_1-\hat{X}_1|  |\Tilde{V}_1 - \hat{V}_1 |\\
                + 2\kappa |\Tilde{V}_1 - \hat{V}_1 | |\mathcal{R}(\Tilde{X}_1,\Tilde{V}_1)|,
        \end{multline}
        and
        \begin{equation*}
            \frac{d}{dt}|\Tilde{V}_1 - \hat{V}_1| \le -\psi_0 \kappa |\Tilde{V}_1 - \hat{V}_1|
            + C\kappa\|\psi\|_{\text{Lip}}   e^{-\kappa\psi_0 t_k} |\Tilde{X}_1-\hat{X}_1|
            + \kappa  |\mathcal{R}(\Tilde{X}_1,\Tilde{V}_1)|.
        \end{equation*}
        Also, with
        \begin{equation*}
            \frac{d}{dt} | \Tilde{X}_1 - \hat{X}_1| \le |\Tilde{V}_1-\hat{V}_1|,
        \end{equation*}
        one has
        \begin{equation*}
             \frac{d}{dt}\left(|\Tilde{X}_1-\hat{X}_1| + |\Tilde{V}_1-\hat{V}_1| \right) \le C\left(|\Tilde{X}_1-\hat{X}_1| + |\Tilde{V}_1-\hat{V}_1| \right) + C|\mathcal{R}|.
        \end{equation*}
        By Gr\"onwall's inequality, it holds
        \begin{equation*}
            |\hat{X}_1-\Tilde{X}_1| + |\hat{V}_1-\Tilde{V}_1| \le C e^{-\kappa\psi_0 t_k} \tau^2.
        \end{equation*}
        Then
        \begin{equation*}
            \left[\mathbb{E}\left[|\hat{X}_1-\Tilde{X}_1|^q + |\hat{V}_1-\Tilde{V}_1|^q\right]\right]^{\frac{1}{q}}
            \le \left[\mathbb{E}\left[|\hat{X}_1-\Tilde{X}_1| + |\hat{V}_1-\Tilde{V}_1|\right]^q\right]^{\frac{1}{q}} \le C e^{-\kappa\psi_0 t_k} \tau^2,
        \end{equation*}
        which implies
        \[
            W_q(\Tilde{S}(\Tilde{f}_{t_k}),\Tilde{f}_{t_{k+1}})\le C(\kappa,\psi_0,\|\psi\|_\text{{Lip}}, R ) e^{-\kappa\psi_0 t_k}\tau^2.
        \]
    \end{proof}

    The estimate on $W_q(\Tilde{S}(\Tilde{f}_{t_k}),\mathcal{Q}_\infty (\Tilde{f}_{t_{k}}) )$ is given by Lemma \ref{lem: 10}.

    \begin{lemma}\label{lem: 10}
        Under the same assumption of Theorem \ref{thm: 2}, it holds
        \begin{equation}\label{eq:lmm10}
            W_q (\mathcal{Q}_\infty (\Tilde{f}_{t_k}),\Tilde{S}(\Tilde{f}_{t_k})) \le C(\kappa, R) \tau^2, \quad q\in [1,\infty).
        \end{equation}
    \end{lemma}
    \begin{proof}
    
    Set $(\hat{X}_1,\hat{V}_1)$ and $\{(X_i,V_i)\}$ to be the solutions of \eqref{eq: tempsys} and \eqref{eq: the RBM limit discrete} respectively. The initial data $(\hat{X}_1(0),\hat{V}_1(0))$ and $\{(X_i(0),V_i(0))\}$ follow $\Tilde{f}_{t_k}.$ Take the coupling $(\hat{X}_1(0),\hat{V}_1(0)) = (X_1(0),V_1(0)).$ We introduce $\{(X_i^E, V_i^E)\}$ as the approximation of $\{(X_i,V_i)\}$ by the forward Euler method with time step $\tau.$ Then, taking the expectation of the selection of $\{(X_j(0),V_j(0)\}_{j\ne 1},$ one has 
    \begin{equation*}
    \begin{aligned}
        \frac{d}{dt} \frac{1}{2}\mathbb{E} |V_1^E - \hat{V}_1|^2 =&
        \int \kappa \psi (| x- X_1(0)|) (v- V_1(0))\cdot \mathbb{E}(V_1^E(t)-\hat{V}_1(t))\Tilde{f}_{t_k}\\
        & - \int \kappa \psi (| x- \hat{X}_1(t)|) (v- \hat{V}_1(t))\cdot \mathbb{E}(V_1^E(t)-\hat{V}_1(t))\Tilde{f}_{t_k}\\
        \le & \kappa \int \psi (|x-X_1(0)|) (\hat{V}_1(t)-V_1(0))\cdot \mathbb{E}(V_1^E(t)-\hat{V}_1(t))\Tilde{f}_{t_k}\\
        &+ \kappa\int\|\psi\|_{\text{Lip}} |X_1(0)-\hat{X}_1(t)| (v-\hat{V}_1(t))\cdot \mathbb{E}(V_1^E(t)-\hat{V}_1(t))\Tilde{f}_{t_k}.\\
    \end{aligned}
    \end{equation*}
    By Lemma \ref{lem: supp} and Lemma \ref{lem: suppS}, one has 
    \[
    | v-\hat{V}_1(t)|\le Ce^{-\kappa\psi_0 t_k}.
    \]
    Note that 
    \[
    |X_1(0)-\hat{X}_1(t)|\le Ce^{-\kappa\psi_0 t_k}\tau
    \]
    and 
    \[
    |\hat{V}_1(t)-V_1(0)|\le Ce^{-\kappa\psi_0 t_k}\tau.
    \]
    Take the expectation of $(X_1(0),V_1(0))$, one has
    \begin{equation*}
         \frac{d}{dt} \mathbb{E} |V_1^E - \hat{V}_1|^2
         \le Ce^{-\kappa\psi_0 t_k}\tau \sqrt{\mathbb{E}|V_1^E - \hat{V}_1|^2}.
    \end{equation*}
    Hence
    \begin{equation*}
         \mathbb{E} |V_1^E - \hat{V}_1|^2 \le Ce^{-2\kappa\psi_0 t_k}\tau^4.
    \end{equation*}
    Similar with \eqref{eq:qqx}, one has
    \begin{equation*}
         \mathbb{E} |X_1^E - \hat{X}_1|^2 \le Ce^{-2\kappa\psi_0 t_k}\tau^6.
    \end{equation*}
    Now consider
    \begin{equation*}
    \begin{aligned}
        \frac{d}{dt}\frac{1}{2}\mathbb{E}|V^E_1(t) -V_1(t)|^2 =& \mathbb{E}(V^E_1(t) -V_1(t))\frac{\kappa}{p-1}\sum\limits_{j=1}^p\bigg[\psi(X_j(0)-X_1(0))(V_j(0)-V_1(0)) \\
        &- \psi(X_j(t) - X_1(t))(V_j(t) - V_1(t))\bigg]\\
        \le& \mathbb{E}\kappa(V^E_1 -V_1)\bigg[\|\psi\|_\text{Lip}|X_2(0)-X_1(0)- X_2(t) + X_1(t))(V_2(0)-V_1(0)) \\&+ \psi(X_2(t) - X_1(t))(V_2(0) - V_1(0) - V_2(t) + V_1(t))\bigg].
    \end{aligned}
    \end{equation*}
    By Lemma \ref{lem: supp} and Lemma \ref{ineq: coarseD}, one knows that $|V_j(t)|\le e^{-\kappa\psi_0 t_k},$ for any $j$ and $t\in [0,\tau].$ Hence it's easy to obtain  
    \begin{equation*}
         \frac{d}{dt} \mathbb{E} |V_1^E - V_1|^2
         \le Ce^{-\kappa\psi_0 t_k}\tau \sqrt{\mathbb{E}|V_1^E - V_1|^2},
    \end{equation*}
    and 
    \begin{equation*}
         \mathbb{E} |V_1^E - V_1|^2 \le Ce^{-2\kappa\psi_0 t_k}\tau^4, \quad \mathbb{E} |X_1^E - \hat{X}_1|^2 \le Ce^{-2\kappa\psi_0 t_k}\tau^6.
    \end{equation*}
    Then, by the triangular inequality of Wasserstein distance, one has
    \begin{equation*}
         W_2 (\mathcal{Q}_\infty (\Tilde{f}_{t_k}),\Tilde{S}(\Tilde{f}_{t_k})) \le C(\kappa, R)e^{-\kappa\psi_0 t_k} \tau^2.
    \end{equation*}
    The estimate of $W_q$ is similar by calculating 
    \begin{equation*}
        \frac{d}{dt}(\mathbb{E}|V_1^E - \hat{V}_1|^q)^{\frac{1}{q}}=\frac{1}{q}(\mathbb{E}|V_1^E - \hat{V}_1|^q)^{\frac{1}{q}-1}\frac{d}{dt}\mathbb{E}|V_1^E - \hat{V}_1|^q \le  C(\kappa, R)e^{-\kappa\psi_0 t_k}\tau.
    \end{equation*}
        
    \end{proof}

    \subsection{Proof of Theorem \ref{thm: 2}}\label{subsec:pfthm2}
    Now we prove Theorem \ref{thm: 2} by combining the above lemmas.
    \begin{proof}[Proof of Theorem \ref{thm: 2}]
    By the triangular inequality, it holds
    \begin{equation*}
        W_q(\mathcal{Q}_\infty^n (f_0), \Tilde{f}(n\tau))  \le  \sum\limits_{m=1}^n W_q (\mathcal{Q}_\infty^{n-m+1} (\Tilde{f}_{m-1}),\mathcal{Q}_\infty^{n-m}(\Tilde{f}_m)).
    \end{equation*}
    And then, by the stability result in Lemma \ref{lem:G stability}, it holds
    \begin{equation*}
        \sum\limits_{m=1}^n W_q (\mathcal{Q}_\infty^{n-m+1} (\Tilde{f}_{m-1}),\mathcal{Q}_\infty^{n-m}(\Tilde{f}_m))
        \le \sum\limits_{m=1}^n C W_q(\mathcal{Q}_\infty (\Tilde{f}_{m-1}), \Tilde{f}_m).
    \end{equation*}
    Again by the triangular inequality, one has
    \begin{equation*}
        W_q(\mathcal{Q}_\infty (\Tilde{f}_{m-1}), \Tilde{f}_m) \le W_q(\mathcal{Q}_\infty (\Tilde{f}_{m-1}), \Tilde{S}(\Tilde{f}_{m-1}))
        + W_q(\Tilde{S}(\Tilde{f}_{m-1}), \Tilde{f}_m).
    \end{equation*}
    Then, by Lemma \ref{lem: w2 S f} and Lemma \ref{lem: 10}, we have
    \begin{equation}\label{eq:3.81}
        \sum\limits_{m=1}^n W_q (\mathcal{Q}_\infty^{n-m+1} (\Tilde{f}_{m-1}),\mathcal{Q}_\infty^{n-m}(\Tilde{f}_m))
        \le \sum\limits_{m=1}^n  C e^{-\kappa\psi_0 (m-1)\tau}\tau^2 \le  C \tau,
    \end{equation}
    where the constant $C$ is independent of $n.$
        One completes the proof.
    \end{proof}

    \begin{remark}
        Although our analysis is based on $W_{q,2}$ distance, the main results in Section \ref{sec:main} also work on $W_{q,p},$ $p\in [1,\infty),$ by noting that
    \begin{equation*}
        \|x \|_p\le d^{1-\frac{1}{q}+\frac{1}{p}}\|x \|_q, \quad x\in\mathbb{R}^d, \quad 1<p,q<\infty.
    \end{equation*}
    \end{remark}

\section{Application: the RBM-gPC}\label{sec:the RBMgPC}

In this section, we show a simple application for the mean-field limit analysis for the Cucker-Smale model combined with the RBM. 
Based on the above analysis, and inspired by Carrillo's work \cite{carrillo2019particle}, we can construct the numerical scheme for the approximation of stochastic mean-field models of swarming which preserve the nonnegativity of the solution. 

Let us consider the stochastic mean-field equation for the distribution function $f(\theta, x, v; t)$:
\begin{equation}\label{eq: liouville}
    \partial_t f + v\cdot \nabla_x f = \nabla_v \cdot [\mathcal{H}[f]f],
\end{equation}
where
\begin{equation*}
    \mathcal{H}[f](\theta,x,v;t) = \kappa \int\int \psi(|x-y|,\theta)(v-w)f(\theta,y,w;t)dwdy,
\end{equation*}
with $\psi$ satisfying  positivity, boundedness, Lipschitz continuity and mononticity conditions: there exists a positive constant $\psi_M  > 0 $  such that 
\begin{equation*}
    0\le \psi(r,\theta) \le \psi_M,\, \forall r\ge0, \quad \|\psi\|_\text{{Lip}}<\infty,
\end{equation*}
\begin{equation*}
    (\psi(r_1,\theta)-\psi(r_2,\theta))(r_1-r_2)\le 0,\quad r_1,r_2 \in \mathbb{R}_+.
\end{equation*}
Here, $\theta$ is a random variable from a measurable space $(\Omega,\mathcal{F})$ to $(I_\Theta,\mathcal{B}_\mathbb{R}),$ with a probability distribution $\pi (\theta),$ with $I_\Theta \subset \mathbb{R}$ and $\mathcal{B}_\mathbb{R}$ the Borel set.

It is easy to see that the total mass and momentum of \eqref{eq: liouville} is conserved in time.
Note that if we directly apply the classical stochastic Galerkin method to \eqref{eq: liouville}, the solution might lose its positivity. However, we can make use of the microscopic dynamics to approximate $f.$ Our method is to consider the corresponding Cucker-Smale model with random input $\theta$ on the interaction kernel $\psi:$
\begin{equation*}
    \left\{\begin{aligned}
        \partial_t x_i(t,\theta) =& v_i(t,\theta),\, i=1,\cdots, N,\\
        \partial_t v_i(t,\theta) =& \frac{\kappa}{N-1}\sum\limits_{j=1}^N 
        \psi(|x_j(t,\theta)-x_i(t,\theta)|,\theta)(v_j(t,\theta)-v_i(t,\theta)).
    \end{aligned}\right.
\end{equation*}

Denote $\{\Phi_k(\theta)\}_{j=0}^\infty$ to be the gPC basis with respect to $\pi,$ which means that $\Phi_k$ is a polynomial of degree $k,$ and
these polynomials form an orthonormal basis of $L^2(\Omega,\mathcal{F},\pi).$ 
We approximate the position and the velocity of $i$-th agent as follows:
\begin{equation*}
    x_i(\theta,t)\approx x^K_i = \sum\limits_{k=0}^{K} {^{k}\hat{x}_i}(t)\Phi_k(\theta),
    \quad v_i(\theta,t)\approx v^K_i = \sum\limits_{k=0}^{K} {}^{k}\hat{v}_i(t)\Phi_k(\theta),
\end{equation*}
where
\begin{equation*}
    ^{k}\hat{x}_i=\int x_i(\theta,t)\Phi_k(\theta)d\pi,\quad ^{k}\hat{v}_i=\int v_i(\theta,t)\Phi_k(\theta)d\pi.
\end{equation*}
We obtain the following system
\begin{equation}\label{eq: particlesG}
    \left\{\begin{aligned}
        \frac{d}{dt} {}^{\ell}\hat{x}_i (t) =& {}^{\ell}\hat{v}_i (t),\\
        \frac{d}{dt} {}^{\ell}\hat{v}_i (t) =& \frac{\kappa}{N-1}\sum\limits_{j=1}^N\sum\limits_{k=0}^K e_{\ell k}^{ij}(^{k}\hat{v}_j(t)-^{k}\hat{v}_i(t)),
    \end{aligned} \quad l = 0,1,\ldots,K.\right.
\end{equation}
where
\begin{equation*}
    e_{\ell k}^{ij}=\frac{1}{|\Phi_\ell|^2}\int \psi(|x_j-x_i|,\theta) \Phi_k(\theta)\Phi_\ell(\theta)d\pi.
\end{equation*}
Then, we design Algorithm \ref{algo:the RBMgPC}, on the generalized Polynomial Chaos expansion in the random space based on the RBM:

\begin{algorithm}\label{algo:the RBMgPC}
\caption{RBM-gPC}
  Consider $N$ samples $(x_i,v_i)$ with $i=1,\cdots,N$ from the initial $f_0(x,v)$, and fix the batch size $p$\;
  Apply the RBM on \eqref{eq: particlesG}: \For{$k=1$ to $T/\tau$} 
  {Divide $\{1,\cdots, N\}$ into $n = N/p$ batches randomly\;
    \For{each batch $\mathcal{C}_q$}
        {
        Update the position and velocity change by
        \begin{equation*}
            \left\{\begin{aligned}
            \frac{d}{dt} {}^{\ell}\hat{x}_i (t) =& {}^{\ell}\hat{v}_i (t),\\
            \frac{d}{dt} {}^{\ell}\hat{v}_i (t) =& \frac{\kappa}{p-1}\sum\limits_{j\in\mathcal{C}_q}\sum\limits_{k=0}^K e_{\ell k}^{ij}(^{k}\hat{v}_j(t)-^{k}\hat{v}_i(t)).
    \end{aligned}\right.
        \end{equation*}
        }
    }
   Reconstruct $\mathbb{E}_\theta[f(\theta,x,v;T)].$
\end{algorithm}

\begin{remark}
    Although being derived from different numerical methods (the RBM and the Monte Carlo method), Algorithm \ref{algo:the RBMgPC} is similar with the Monte Carlo gPC (MCgPC) proposed in \cite{carrillo2019particle} at first glance. Both can reduce the total cost of \eqref{eq: particlesG} to $\mathcal{O}(KN).$ The main difference between them is that the RBM-gPC preserves the momentum during evolution. Actually, since the RBM updates information in the unit of batches, (instead, the MCgPC updates information of each particle independently at per time step.) for any $t\in[t_m,t_{m+1}],$ the total sum of velocities is conserved in each batch by calculus:
    \begin{equation*}
        \frac{d}{dt} \sum\limits_{i\in \mathcal{C}_q} {}^{\ell}\hat{v}_i (t) = \frac{\kappa}{p-1}\sum\limits_{i,j\in\mathcal{C}_q}\sum\limits_{k=0}^K e_{\ell k}^{ij}(^{k}\hat{v}_j(t)-^{k}\hat{v}_i(t))=0,
    \end{equation*}
    for $l=0,\cdots,K.$ Furthermore, it holds 
    \begin{equation*}
        \sum\limits_{i=1}^N {}^{\ell}\hat{v}_i (t) = \sum\limits_{q=1}^{N/p} \sum\limits_{i\in \mathcal{C}_q}{}^{\ell}\hat{v}_i (t) = \sum\limits_{q=1}^{N/p} \sum\limits_{i\in \mathcal{C}_q}{}^{\ell}\hat{v}_i (t_m) = \sum\limits_{q=1}^{N/p} \sum\limits_{i\in \mathcal{C}_q}{}^{\ell}\hat{v}_i (t_{0}) = \sum\limits_{i=1}^N {}^{\ell}\hat{v}_i (0).
    \end{equation*}
    
\end{remark}




\begin{remark}
    We use $$x_i^{K,R}(t,\theta):= \sum\limits_{k=0}^{K} {^{k}\hat{x}_i}(t)\Phi_k(\theta),\quad v_i^{K,R}(t,\theta):=\sum\limits_{k=0}^{K} {}^{k}\hat{v}_i(t)\Phi_k(\theta)$$ to denote the approximation of position and velocity obtained by Algorithm \ref{algo:the RBMgPC}. The kinetic distribution $f^R_{N,K}(\theta,x,v,t)$ is recovered from the empirical density function $$f^R_{N,K}(\theta,x,v,t) = \frac{1}{N}\sum\limits_{i=1}^N \delta(x-x_i^{K,R}(t,\theta))\otimes\delta(v-v_i^{K,R}(t,\theta)).$$ Common approximations methods for Dirac delta distributions can be used when needed. For example, we fix an upper and lower bound for the $(x,v)$-space, to restrict our computational domain in a cube space and discretize it in uniform cells $\{C_l\}_{l=1}^{N_l}$ of size $|C_l|=h^{2d},$ where $h$ is the edge length of each cell. One can reconstruct the distribution $f^R_{N,K}$ by histograms counting the number of particles belonging to each cell. Then the discrete distribution function $\Bar{f}^R_{N,K}$ reads
    \begin{equation}\label{eq: approx f}
       \Bar{f}^R_{N,K}(\theta,x,v,t) = \frac{1}{|C_l|}\left(\frac{1}{N}\sum\limits_{i=1}^N \mathbb{I}(x_i^{K,R}(t,\theta),v_i^{K,R}(t,\theta)\in C_l)\right)\cdot \left(\sum\limits_{l=1}^{N_l}\mathbb{I}((x,v)\in C_l)\right),
    \end{equation}
    where $\mathbb{I}(\cdot)$ is the indicator function.
    Clearly, the positivity of the distribution function is preserved. One can also refer to a brief review, by Lee et al. \cite{lee2012regularized}, on the numerical methods for regularized Dirac delta functions.
\end{remark}


\begin{remark}
    With respect to the stochastic variable $\theta,$ the dynamics of the $N$-particle system achieves spectral accuracy as we adopt the generalized polynomial chaos (gPC) basis \cite{xiu2002wiener}, provided that we have a smooth dependence of the particle solution from the random field.

    For any given $\theta$ and $K,$ denote the distribution function of the $N$-particle system \eqref{eq: particlesG} by $f_{N,K}.$ 
    If one uses the reconstruction \eqref{eq: approx f}, it holds that
        \begin{equation*}
             \mathbb{E} W_2 (f^R_{N,K},\Bar{f}^R_{N,K}) \le  \left(\mathbb{E}\frac{1}{N}\sum\limits_{i=1}^N |(x_i^R,v_i^R) - c_{l(i)}|^2 \right)^{\frac{1}{2}},
        \end{equation*}
        where $c_{l(i)}$ is the center of the cell $C_l$ containing $(x_i^R,v_i^R).$ Since $h$ is the uniform edge length, $|(x_i^R,v_i^R) - c_{l(i)}|^2 \le \sum\limits^d (h/2)^2$ for any $i.$ Combining with the estimate of $\E W_2 (f_{N,K},f^R_{N,K})$ in Section \ref{subsec: meanCS}, it holds that
        \begin{equation*}
        \mathbb{E} W_2 (f_{N,K},\Bar{f}^R_{N,K})     \le \mathcal{O}(\sqrt{\frac{\tau}{p-1}}) +
       \left\{\begin{aligned}
            &\mathcal{O}(N^{-\frac{1}{4}}) \, &\text{if } d<2,\\
        &\mathcal{O}(N^{-\frac{1}{4}}\sqrt{\log(1+N)}) \, &\text{if } d=2,\\
        &\mathcal{O}(N^{-\frac{1}{d}}) \, &\text{if } d>2,
        \end{aligned}\right.
        + \mathcal{O}(\sqrt{d}h).
    \end{equation*}
\end{remark}

\subsection{Numerical experiments}
Here we present some numerical tests for the RBM-gPC. In the following, we take the Legendre gPC base.

\paragraph{The space homogeneous case}
We consider the space-independent case in the one-dimensional setting:
\begin{equation}\label{case: baby}
    \partial_t f(\theta,v;t) = \partial_v \left[ K(\theta)(v-u)f(\theta,v;t)\right],
\end{equation}
with $u = \int vf(v)dv$ and $K(\theta)=0.5 + 0.01\theta,$ $\theta\sim \mathcal{U}(0,1).$ The long time solution is a Dirac delta $\delta(v-u)$ provided $K(\theta)>0,$ see \cite{tosin2017boltzmann,carrillo2010asymptotic}. We take the Legendre polynomial basis. The gPC approximation is given for $h=0,\cdots,M$ by 
\begin{equation}\label{eq: fgPC}
    \partial_t \hat{f}_h (v;t) = \frac{1}{|\Phi_h|^2} \partial_v \left[ \sum\limits_{k=0}^M (v-u) \mathcal{H}_{hk}\hat{f}_{t_k}(v;t)\right],
\end{equation}
where
\begin{equation*}
    \mathcal{H}_{hk} = \int K(\theta)\Phi_h (\theta) \Phi_k (\theta) d g (\theta), \quad 
    \hat{f}_h (v;t) = \int f(\theta,v;t) \Phi_h (\theta) d g (\theta).
\end{equation*}
We consider an initial density function $f_0(v)$ by
\begin{equation*}
    f_0(v) = \beta\left[\exp{\left(-\frac{(v-\mu)^2}{2\sigma^2}\right)} + \exp{\left(-\frac{(v+\mu)^2}{2\sigma^2}\right)} \right], \quad \sigma^2=0.1, \,\mu = 0.5,
\end{equation*}
with $\beta>0$ a normalization constant. Discrete samples of the initial data of the ODE system are obtained from $f_0(v).$ Denote the expected temperature of the system $\mathcal{T}$ by 
\begin{equation*}
    \mathcal{T}(f) = \int (v-u)^2 f(\theta,v;t)dvdg(\theta).
\end{equation*}
In Figure \ref{fig:babyshape}, we show the evolution of the expected density toward the Dirac delta function $\delta(v-u),\,u=0$ for the RBM-gPC scheme at different times with $p=2$ and $N=10^4.$
\begin{figure}[htbp]
    \centering
    \includegraphics[scale=0.5]{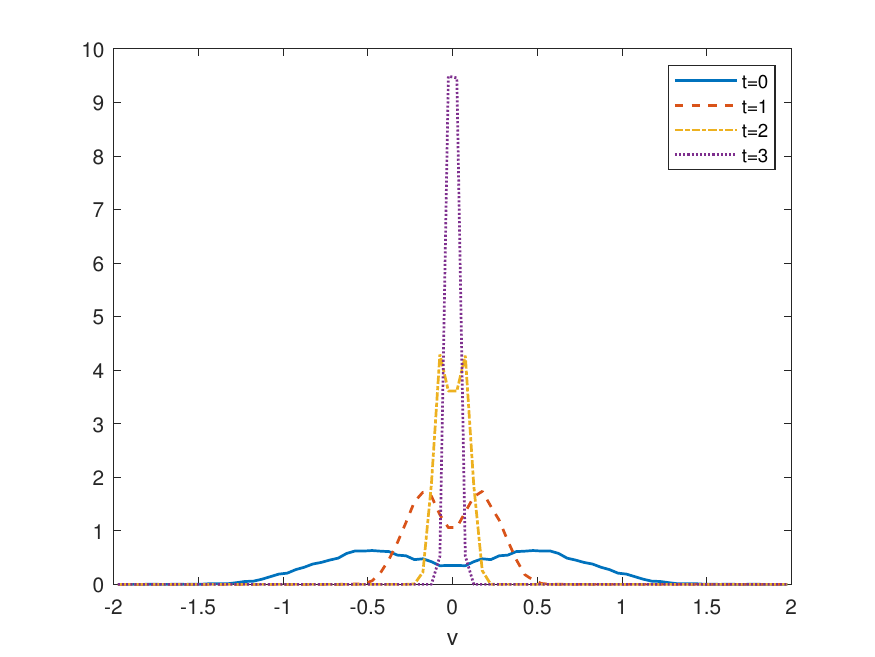}
    \caption{Evolution of the expected density $\delta(v-u),\,u=0$ for the RBM-gPC scheme.}
    \label{fig:babyshape}
\end{figure}
In Figure \ref{fig:tp}, we take $T=0.5,$ $K=3,$ $dt=10^{-2}$ and $dv=10^{-2}.$ Time integration is performed through a $4th$ order Runge-Kutta method. In Figure \ref{fig:Ntp}, we evaluate the mean square error (MSE) for the expected temperature 
\begin{equation*}
    \text{MSE}_\text{T} := \frac{1}{n_m}\sum\limits_{i=1}^{n_m}\left|\mathcal{T}(f_{\text{ref}}) - \mathcal{T}(f_{\text{rbm}}^{(i)}) \right|^2,
\end{equation*}
as $N$ grows with $p=2,$ where $\mathcal{T}(f_{\text{ref}})$ denotes the reference expected temperature obtained by \eqref{eq: fgPC} through a central difference scheme, and $\mathcal{T}(f_{\text{rbm}}^{(i)})$ denotes the expected temperature obtained by the RBM-gPC in the $i$-th turn. We have applied $n_m=100$ for random initialization and batch splits. 
Figure \ref{fig:tvtp} shows the difference of TV distance between the expected reference distribution and the expected distribution by the RBM-gPC:
\begin{equation*}
    \text{Err}_\text{TV} := \left| \mathbb{E}_\theta f_{\text{ref}} - \mathbb{E}_\theta \frac{1}{n_m}\sum\limits_{i=1}^{n_m} f_{\text{rbm}}^{(i)} \right|_{L^1},
\end{equation*}
with $n_m =100$ as well.


\begin{figure}[htbp]
    \centering
    \subfigure[$\text{MSE}_\text{T}$]{
    \begin{minipage}[t]{0.48\textwidth}
        \centering
        \includegraphics[scale=0.46]{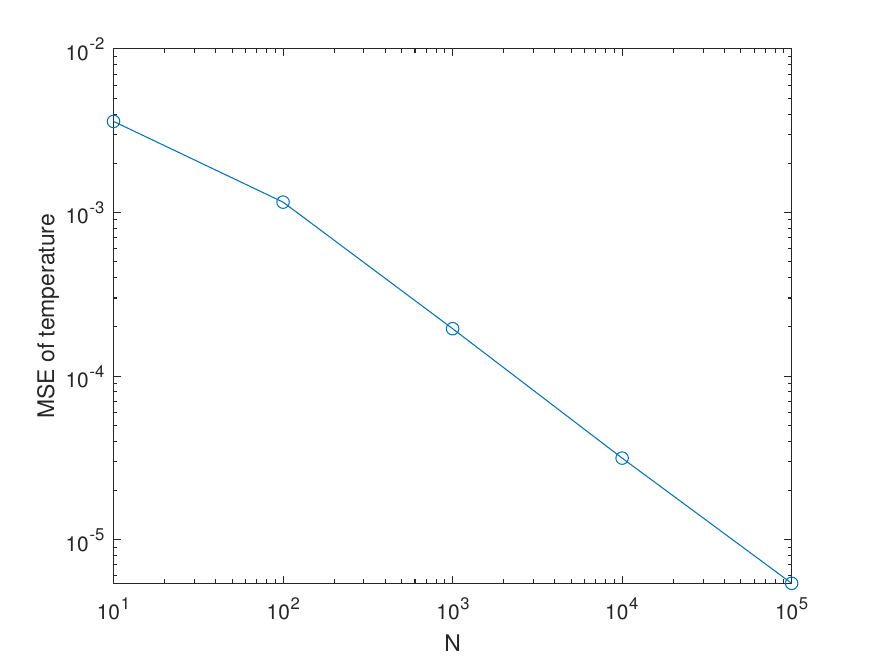}
    \label{fig:Ntp}
    \end{minipage}}
    \subfigure[$\text{Err}_\text{TV}$]{
    \begin{minipage}[t]{0.48\textwidth}
        \centering
        \includegraphics[scale=0.46]{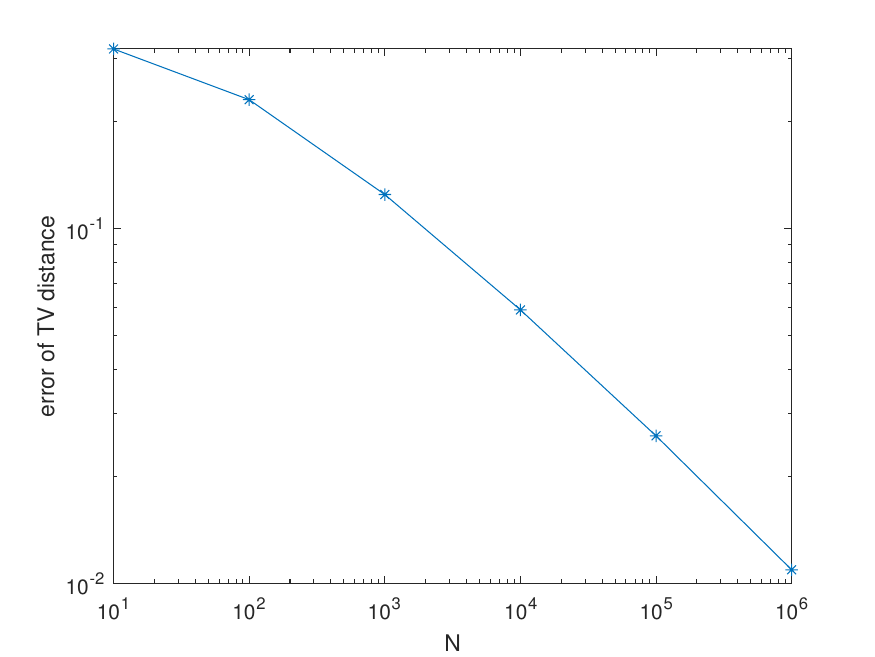}
        \label{fig:tvtp}
        \end{minipage}}
        \caption{(a): The $\text{MSE}_\text{T}$ of expected temperature between reference solution and the RBM-gPC solution. (b): The $\text{Err}_\text{TV}$ between the expected reference distribution and the expected distribution by the RBM-gPC. Both are obtained at $T=0.5,$ with $dt=dv=10^{-2},$ $p=2.$}
        \label{fig:tp}
\end{figure}
Also, Figure \ref{fig:msepp} and Figure \ref{fig:msedt} show the decay of $\text{MSE}_\text{T}$ as the batch size $p$ and the time step $dt$, respectively. The reference distribution of Figure \ref{fig:msepp} is obtained by the RBM-gPC with $N=2^8,$ $p = 2^8,$ and $dt = 10^{-2},$ and of Figure \ref{fig:msedt} is obtained with $N=2^8,$ $p = 2,$ and $dt = 10^{-2}.$ We still take $n_m=100.$
\begin{figure}[htbp]
    \centering
    \subfigure[]{
    \begin{minipage}[t]{0.48\textwidth}
        \centering
        \includegraphics[scale=0.46]{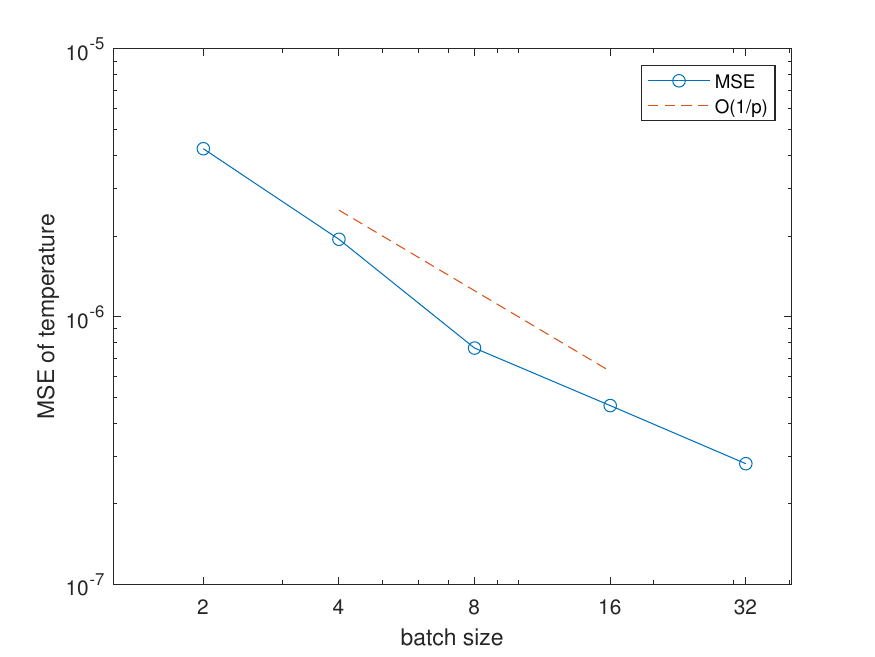}
    \label{fig:msepp}
    \end{minipage}}
    \subfigure[]{
    \begin{minipage}[t]{0.48\textwidth}
        \centering
        \includegraphics[scale=0.46]{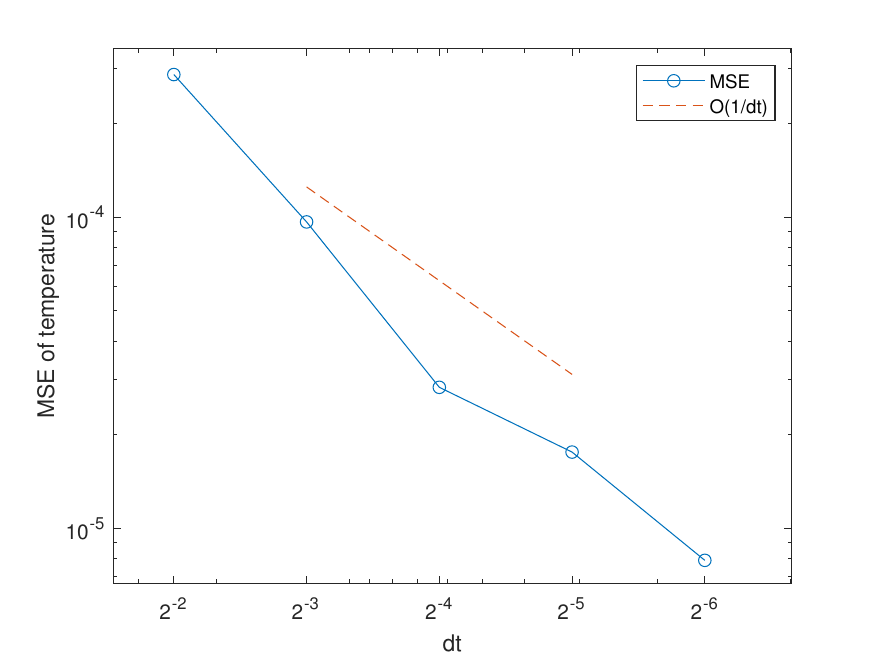}
    \label{fig:msedt}
    \end{minipage}}
    \caption{(a):The $\text{MSE}_\text{T}$ of expected temperature of the RBM-gPC solution, with $dt=10^{-2}$ and batch size $p$ varying. (b): The $\text{MSE}_\text{T}$ of expected temperature of the RBM-gPC solution, with $p=2$ and time step $dt$ varying. Both are obtained at $T=0.5$ with $N=2^8.$}
\end{figure}

\paragraph{Stochastic 1D Cucker-Smale dynamics}
In this test, we consider the 1D Cucker-Smale dynamics. Let us consider the initial distribution as the following bivariate and bimodal distribution
\begin{equation*}
    f_0(x,v) = \frac{1}{2\pi \sigma_x\sigma_v}\exp\left(-\frac{x^2}{2\sigma_v ^2}\right)\left[\exp{\left(-\frac{(v+1)^2}{2\sigma_x ^2}\right)}+ \exp{\left(-\frac{(v-1)^2}{2\sigma_x^2}\right)}\right],
\end{equation*}
with $\sigma_x^2 =0.5,$ $\sigma_v^2=0.2.$ Our initial data for particle positions and velocities are
sampled from $f_0.$ The stochastic interactions are given by
\begin{equation}\label{eq:Hkernel}
    H(\theta; |x_i - x_j|) = \frac{1}{(1+|x_i - x_j|^2)^{\gamma(\theta)}},
\end{equation}
with $\gamma (\theta) = 0.1+ 0.05 \theta, \theta \sim \mathcal{U}[-1,1].$ The results are obtained through the RBM-gPC scheme with $dt = 10^{-2},$ $N=10^4$, batch size $p =2,$ and $K=3.$ The expected distribution is reconstructed by \texttt{ksdensity} in Matlab in the domain $[-3,3]\times[-3,3]$ with 100 gridpoints in both space and velocity. Figure \ref{fig:1d} presents the evolution over the time interval $t\in[0,4]$ of the expected distribution.

\begin{figure}[htbp]
    \centering
    \includegraphics[scale = 0.36]{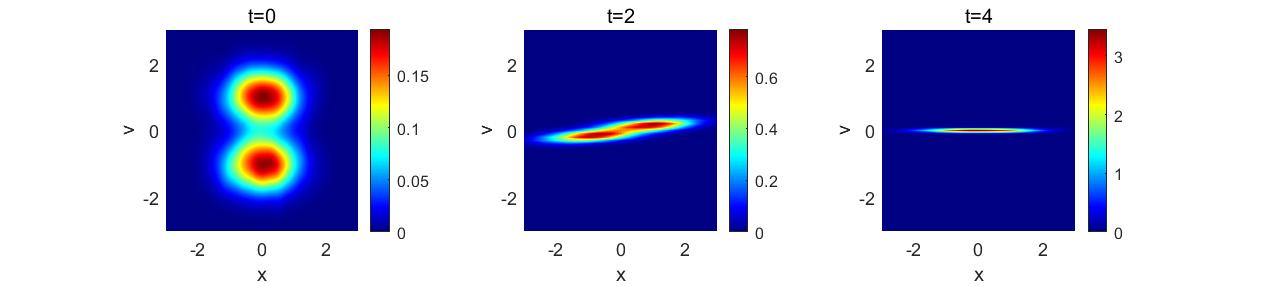}
    \caption{The expected distribution of the stochastic 1D Cucker-Smale dynamics with $N=10^4$ agents.}
    \label{fig:1d}
\end{figure}

\paragraph{Stochastic 2D Cucker-Smale dynamics}
Figure \ref{fig:2d} shows the evolution over the time interval $t\in [0,4]$ of the 2D Cucker-Smale nonhomogeneous mean-field model. For the initial distribution, we consider uniformly distributed $N=10^4$ particles on a 2D annulus with a circular counterclockwise motion
\begin{equation*}
    f_0 (x,v) = \frac{1}{|C|}\chi (x\in C)\delta \left( v - \frac{k\wedge x}{|x|}\right),
\end{equation*}
being $C: = \{x\in \mathbb{R}^2: 0.5\le |x_1-x_2|\le 1 \}$ and $k$ the fundamental unit vector of the $z$-axis, where $|C|$ means the cardinality of set $C.$
The interaction kernel shares the same form with \eqref{eq:Hkernel}.
The evolution shows the flocking phenomenon in the top row and that the mean velocity is conserved in the bottom row of Figure \ref{fig:2d}. The results are obtained through the RBM-gPC scheme with $dt = 10^{-2},$ $N=10^4$, batch size $p =2,$ and $K=3.$ The reconstruction follows the same method as in the 1D case.
\begin{figure}[htbp]
    \centering
    \includegraphics[scale=0.35]{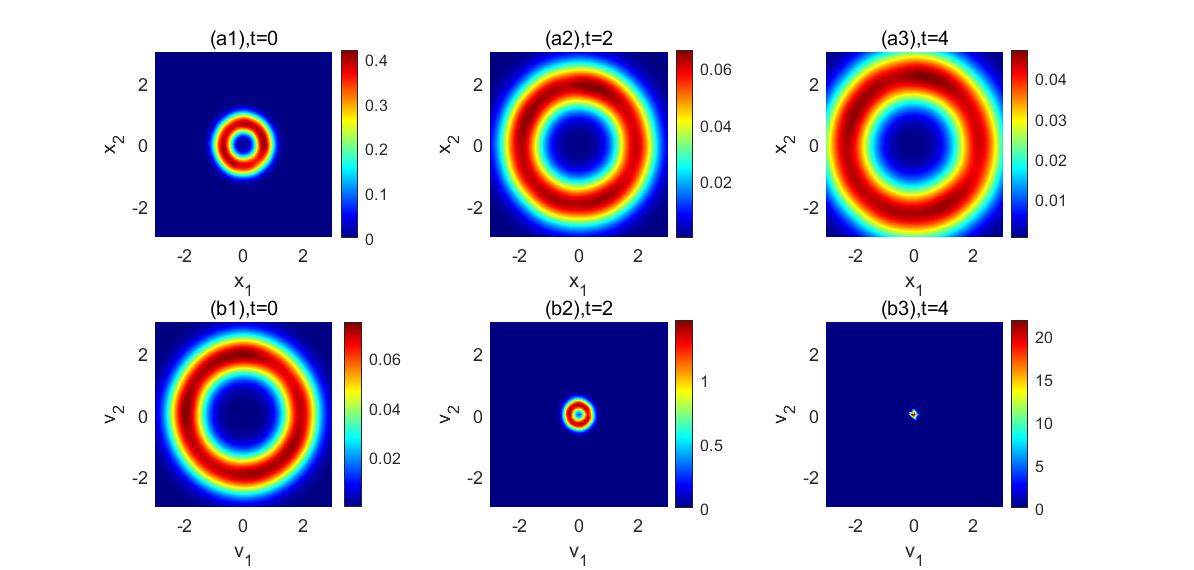}
    \caption{The expected distribution of the stochastic 2D Cucker-Smale dynamics with $N=10^4$ agents.}
    \label{fig:2d}
\end{figure}

\section{Conclusion}\label{sec:conclusion}
In conclusion, we give an analysis on the mean-field limit of the RBM for the Cucker-Smale model and obtain a long-time approximation. In addition, as a simple application, we propose a generalized Polynomial Chaos based Random Batch Method (RBM-gPC) as a numerical method to simulate the mean-field limit dynamics of the Cucker-Smale model with random inputs.

This paper focuses on the separate limits of the number of particles $N$ and the time step $\tau$ but not a uniform limit of $(N,\tau).$ For future work, it will be interesting to investigate the mixture of particles as time evolves, and to study on the uniform limit. Also, numerically, an error estimate on this model applied with the particle method (corresponding to a new model) can be an interesting topic to study.

\section*{Acknowledgement}
We thank Prof. Shi Jin and Prof. Lei Li for detailed and valuable comments and advice on the manuscript. This work is supported by the NSFC grant (No. 12031013, 12201404), the Shanghai Municipal Science and Technology Key Project (No. 22JC1402300), and Project supported by the National Science Foundation for International Senior Scientists (No. 12350710181).

\begin{appendix}
    
    \section{Proof of Lemma \ref{ineq: coarseD}}\label{ass:A}
    The proof is given by Lemma 2.4 in \cite{ha2021uniform}.
    \begin{proof}[Proof of Lemma \ref{ineq: coarseD}]
        For the first assertion, we claim that the relative velocities are non-increasing in time. Let $t\in [t_{m-1}, t_m)$ be given. Then one can choose time-dependent indices $k$ and $\ell$ such that 
        \begin{equation*}
            |V_k^R (t) - V_\ell^R (t)| = \mathcal{D}_{V^R}(t).
        \end{equation*}
        Then for such $k$ and $\ell,$ one has 
        \begin{equation}\label{eq:A1}
        \begin{aligned}
            \frac{d}{dt} | V_k^R(t) -V_\ell^R(t)|^2 = &
            2(V_k^R(t) - V_\ell^R (t)) \cdot\frac{d}{dt}(V_k^R(t) - V_\ell^R (t))\\
            = & \frac{2\kappa}{p-1} \sum\limits_{j\in [k]_m}\psi (|X_j^R - X_k^R|)(V_j^R - V_k^R)\cdot(V_k^R - V_\ell^R)\\
            &- \frac{2\kappa}{p-1}\sum\limits_{j\in[i]_m}\psi (|X_j^R - X_\ell^R|)(V_j^R - V_\ell^R)\cdot(V_k^R - V_\ell^R).
        \end{aligned}
        \end{equation}
        In order to show that the right-hand side of \eqref{eq:A1} is not positive, we use the maximality of $| V_k^R(t) -V_\ell^R(t)|$ at time $t.$ Since
        \begin{equation*}
            | V_k^R(t) -V_\ell^R(t)| \ge | V_j^R(t) -V_\ell^R(t)|, \quad j = 1,\cdots,N,
        \end{equation*}
        one has
        \begin{multline*}
            (V_j^R(t)-V_k^R (t))\cdot(V_k^R(t)-V_\ell^R (t)) =
            - ((V_k^R-V_\ell^R) - (V_j^R-V_\ell^R))\cdot (V_k^R - V_\ell^R)\\
            \le - | V_k^R(t) -V_\ell^R(t)|^2 + | V_j^R(t) -V_\ell^R(t)|| V_k^R(t) -V_\ell^R(t)|\le 0, \quad j= 1,\cdots,N.
        \end{multline*}
        Similarly, one has 
        \begin{equation*}
            (V_j^R(t)-V_\ell^R (t))\cdot (V_k^R(t)-V_\ell^R (t)) \ge 0, \quad j=1,\cdots,N.
        \end{equation*}
        Hence, $| V_k^R(t) -V_\ell^R(t)|^2$ is non-increasing in time.

        Then by the definition od $\mathcal{D}_{X^R}$ and $\mathcal{D}_{V^R},$ one has 
        \begin{equation*}
            \frac{d}{dt}\mathcal{D}_{X^R}(t)\le \mathcal{D}_{V^R}(t) \le\mathcal{D}_{V^R}(0), \quad a.e.\, t>0.
        \end{equation*}
        This yields the second assertion on $\mathcal{D}_{X^R}.$
    \end{proof}

    \section{Details on Proposition \ref{lem: 10old}}\label{app:semigroup}

    \paragraph{Fixed-time regularity of $\Tilde{f}$.} First we check the regularity of $\Tilde{f}$ by Lemma \ref{lem: C2}, which is needed in the proof of Proposition \ref{lem: 10old}.
    
    \begin{lemma}\label{lem: C2}
    For the flocking kinetic model \eqref{eq: mfl for cs}, suppose that the initial probability measure $f_0 \in (C^2\cap W^{2,\infty})(\mathbb{R}^{2d})$ is compactly supported, and the interaction kernel $\psi$ satisfies $$\sup\limits_{r\ge 0} (|\psi^\prime(r)| + |\psi^{\prime\prime}(r)|) \le \epsilon_\psi$$ with a constant $\epsilon_\psi.$ Then, for any $T\in (0,\infty),$ there exists a unique classical solution $\Tilde{f}_t \in C^2([0,T)\times \mathbb{R}^{2d})$ for \eqref{eq: mfl for cs}. 
\end{lemma}
    The proof stem of Lemma \ref{lem: C2} is generally a smoothness check for $\Tilde{f}$ inspired by in [\cite{Ha2008particle}, Theorem 3.1]. The main difference with \cite{Ha2008particle} is that here we need better regularity for the operator $(\mathcal{\Bar{L}}^*)^2$ and $(\mathcal{\Tilde{L}}_k^*)^2$ in Lemma \ref{lem: 10}.
    \begin{proof}[Proof of Lemma \ref{lem: C2}.]
    Local existence follows from [\cite{Ha2008particle} Theorem 3.1] by the standard fixed point argument. One only needs to obtain a $C^2$ bound of $f$ to conclude the global existence of classical solutions. One can express the non-conservative kinetic model \eqref{eq: mfl for cs} in terms of the non-linear transport operator $\mathcal{T}:=\partial_t + v\cdot \nabla_x + \xi [\Tilde{f}] \cdot \nabla_v,$
    $$ \mathcal{T}\Tilde{f} = -\Tilde{f}\nabla_v \cdot \xi[\Tilde{f}].$$
    We claim that there exists a positive constant $C(d,\kappa,\epsilon_\psi, R)$ such that
    \begin{equation}\label{eq:C2-1}
        \left | \mathcal{T}\Tilde{f} \right| \le C |\Tilde{f}|,
    \end{equation}
    \begin{equation}\label{eq:C2-2}
        \left | \mathcal{T}(\partial_{x_i}\Tilde{f}) \right| \le C (|\Tilde{f}| + |\nabla_v \Tilde{f}| + |\partial_{x_i}\Tilde{f})|,
    \end{equation}
    \begin{equation}\label{eq:C2-3}
        \left | \mathcal{T}(\partial_{v_i }\Tilde{f}) \right| \le C (|\nabla_v \Tilde{f}| + |\partial_{x_i}\Tilde{f}|),
    \end{equation}
    \begin{equation}\label{eq:C2-4}
        \left | \mathcal{T}(\partial_{x_i x_j}\Tilde{f}) \right| \le C (  |\partial_{x_i}\Tilde{f}| + |\partial_{x_j}\Tilde{f}| + |\partial_{x_i x_j}\Tilde{f}| +|\nabla_v \Tilde{f}| + |\nabla_v\partial_{x_i}\Tilde{f}| + |\nabla_v\partial_{x_j}\Tilde{f}|),
    \end{equation}
    \begin{equation}\label{eq:C2-5}
        \left | \mathcal{T}(\partial_{x_i v_j}\Tilde{f}) \right| \le C ( |\partial_{x_i x_j}\Tilde{f}| +|\nabla_v \Tilde{f}| + |\partial_{v_j}\partial_{x_i}\Tilde{f}| + |\nabla_v\partial_{v_j}\Tilde{f}| + |\nabla_v\partial_{x_i}\Tilde{f}|),
    \end{equation}

    \begin{equation}\label{eq:C2-6}
        \left | \mathcal{T}(\partial_{v_i v_j}\Tilde{f}) \right| \le C ( |\partial_{x_i v_j}\Tilde{f}| + |\partial_{x_j}\partial_{v_i}\Tilde{f}| +  |\partial_{v_j}\partial_{v_i}\Tilde{f}|).
    \end{equation}

    To verify these inequalities, note that $-\nabla_v\cdot \xi [\Tilde{f}] = \kappa d \int \psi (|y-x|)\Tilde{f}(y,w,t)dydw$ and thus $|\nabla_v\cdot \xi [\Tilde{f}]| \le \kappa d.$ By direct calculus, one obtains the following estimates:\\~\\
    For \eqref{eq:C2-1},
    \begin{equation*}
         \left | \mathcal{T}\Tilde{f} \right| \le  \left| -\kappa \Tilde{f} \nabla_v \cdot \int \psi(|y-x|) (w-v) \Tilde{f}(y,w,t) dydw\right| \le \kappa d |\Tilde{f}|;
    \end{equation*}
    for \eqref{eq:C2-2},
    \begin{equation*}
    \begin{aligned}
        \left | \mathcal{T}(\partial_{x_i}\Tilde{f}) \right| &\le \left| - (\partial _{x_i} \xi [\Tilde{f}]) \cdot \nabla_v \Tilde{f} -  (\partial_{x_i} \nabla_v \cdot \xi [\Tilde{f}])\Tilde{f}  -  (\nabla_v \cdot \xi [\Tilde{f}])\partial _{x_i} \Tilde{f} \right|\\
        &\le \kappa \epsilon_\psi (\sqrt{M_2^v f_0} + 2R) |\nabla_v \Tilde{f}| + \kappa d \epsilon_\psi |\Tilde{f}| + \kappa d |\partial_{x_i} \Tilde{f}|)\\
        & \le C (|\Tilde{f}| + |\nabla_v \Tilde{f}| + |\partial_{x_i}\Tilde{f})|;
    \end{aligned}
    \end{equation*}
    For \eqref{eq:C2-3},
    \begin{equation*}
        |\mathcal{T}\partial_{v_i} \Tilde{f} | \le |-\partial_{x_i} \Tilde{f}| + \kappa |\nabla_v \Tilde{f}| + \kappa d |\partial_{v_i} \Tilde{f}|
        \le C (|\nabla_v \Tilde{f}| + |\partial_{x_i}\Tilde{f}|).
    \end{equation*}
    The estimate on \eqref{eq:C2-4}-\eqref{eq:C2-6} is similar and we omit it for brevity.
    
    Now, let $F(t)$ measure the $W^{2,\infty}$-norm of $\Tilde{f},$
    $$ F := \|\Tilde{f}(\cdot,t)\| _{W^{2,\infty}} .$$
    The above inequalities imply
    $$ \frac{d}{dt}F(t) \le CF(t).$$
    Then we end up with the energy bound by Gr\"onwall's inequality
    $F(t)\le F(0) e^{Ct}.$
    Equipped with this a priori $W^{2,\infty}$ estimate, standard continuation principle yields a global extension of local classical solutions. 
\end{proof}

The following is a simple prerequisite for Proposition \ref{lem: 10old}.

    \begin{lemma}\label{lem: 4.2}
For any $g_1\in L^1(\mathbb{R}^{2d}),g_2\in L^1(\mathbb{R}^{2dp}),$ the operators $\mathcal{\Tilde{L}}_k^*$ and $\mathcal{\Bar{L}}^*$ defined by \eqref{op:tLk} and \eqref{op:bL} respectively, it holds
    \begin{equation}\label{eq: 4.2.1}
         \int | e^{\tau \mathcal{\Tilde{L}}_k^*}g_1| dz \le  \int | g_1| dz,
    \end{equation}
    \begin{equation}\label{eq: 4.2.2}
        \int \left|  e^{\tau \Bar{\mathcal{L}}^*}  g_2  \right| dz_1 \cdots dz_p \le \int \left|  g_2  \right| dz_1 \cdots dz_p.
    \end{equation}
\end{lemma}
    \begin{proof} 
    One only needs to consider $g_1\ge 0$ since there is composition $g_1 = g_1^+ -g_1^-.$ Also, since $e^{\tau \mathcal{L}_k^*}cg_1 = ce^{\tau \mathcal{L}_k^*}g_1$ and $e^{\tau \Bar{\mathcal{L}}^*}cg_2 = ce^{\tau \Bar{\mathcal{L}}^*}g_2$ for any constant $c,$ one can set $g_1$ be a probability density function. Then,
    $$\int |e^{\tau \mathcal{\Tilde{L}}_k^*}g_1| dz = \int e^{\tau \mathcal{\Tilde{L}}_k^*}g_1 dz,$$
    since $e^{\tau \mathcal{\Tilde{L}}_k^*}g_1$ is another probability density function derived from $g_1$ by the dynamics \eqref{eq: tempsys}. It's clear that
    $$\int e^{\tau \mathcal{\Tilde{L}}_k^*}g_1 dz = 
    \int\int_0^\tau \frac{\partial}{\partial t} \Tilde{g}_1(t)dtdz + \int g_1 dz = \int g_1dz,$$
    where $\Tilde{g}_1(t)$ denotes the solution of \eqref{eq: mfl for cs} with initial datum $g_1.$ For the general $g_1,$ 
    \begin{equation*}
         \int | e^{\tau \mathcal{\Tilde{L}}_k^*}g_1| dz \le  
         \int | e^{\tau \mathcal{\Tilde{L}}_k^*}g_1^+| dz
         + \int | e^{\tau \mathcal{\Tilde{L}}_k^*}g_1^-| dz
         =  \int | g_1^+| dz + \int | g_1^-| dz =  \int | g_1| dz.
    \end{equation*}
    The proof of \eqref{eq: 4.2.2} is similar.
\end{proof}

    \paragraph{Proof of Proposition \ref{lem: 10old}.}
    It's clear that $\Tilde{S} = e^{\tau \mathcal{\Tilde{L}}_k^*}\Tilde{f}_{t_k},$ where 
    \begin{equation}\label{op:tLk}
        \mathcal{\Tilde{L}}_k^* := -  (v \cdot \nabla_x) \cdot -
    \nabla_v \cdot \left( \int \kappa\psi(|y-x|)(w-v)f_{t_k}(y,w)dydw \cdot \right).
    \end{equation}
    Define the operator of \eqref{the RBM Liouville} as
    \begin{equation}\label{op:bL}
        \Bar{\mathcal{L}}^* := -\sum\limits_{i=1}^p (v_i\nabla_{x_i})\cdot  - \sum\limits_{i=1}^p\nabla_{v_i}\cdot (\xi_i \cdot).
    \end{equation}
    Then one has
    \begin{equation*}
    \begin{aligned}
        \mathcal{Q}_\infty (\Tilde{f}_{t_k}) =& \int e^{\tau\Bar{\mathcal{L}}^*} \Pi_{i=1}^p \Tilde{f}_{t_k}(x_i,v_i)dx_2\cdots dx_pdv_2\cdots dv_p \\
        =&\Tilde{f}_{t_k} (x_1,v_1) + \tau \int \Bar{\mathcal{L}}^* \Pi_{i=1}^p \Tilde{f}_{t_k} (x_i,v_i)dx_2\cdots dx_pdv_2\cdots dv_p \\
        &+ \int_0^\tau (\tau -s) \int (\Bar{\mathcal{L}}^*)^2 e^{s\Bar{\mathcal{L}}^*} \Pi_{i=1}^p \Tilde{f}_{t_k} (x_i,v_i)dx_2\cdots dx_pdv_2\cdots dv_p ds,
    \end{aligned}
    \end{equation*}
    where
    \begin{small}
    \begin{equation*}
        \begin{aligned}
            \int \Bar{\mathcal{L}}^* \Pi_{i=1}^p \Tilde{f}_{t_k} (z_i) dz_2\cdots dz_p 
            = & -\sum\limits_{i=1}^p \int (v_i \nabla_{x_i})\cdot \Pi_{j=1}^p \Tilde{f}_{t_k} (z_j) - \nabla _{v_i} \cdot (\xi_i \Pi_{j=1}^p)\Tilde{f}_{t_k} (z_j)dz_2\cdots dz_p)\\
            = & - (v_1 \nabla_{x_1}) \cdot \Tilde{f}_{t_k} (z_1) - \nabla_{v_1} \cdot \int \kappa \psi (|y-x_1|)(w-v_1) \Tilde{f}_{t_k} (y,w) dydw.
        \end{aligned}
    \end{equation*}
    \end{small}
    Note that $\Tilde{S}(\Tilde{f}_{t_k}) = \Tilde{f}_{t_k} + \tau \Tilde{\mathcal{L}}_k^*\Tilde{f}_{t_k} + \int_0^\tau (\tau -s)(\Tilde{\mathcal{L}}_k^*)^2 e^{s\Tilde{\mathcal{L}}_k^*} \Tilde{f}_{t_k} ds,$ then
    \begin{small}
        \begin{equation*}
        \left| \Tilde{S}(\Tilde{f}_{t_k}) - \mathcal{Q}_\infty (\Tilde{f}_{t_k}) \right|(z_1)\le \int _0^\tau (\tau-s) \left[ \left|(\Tilde{\mathcal{L}}_k^*)^2 e^{s\Tilde{\mathcal{L}}_k^*} \Tilde{f}_{t_k}(z_1)\right|  +  \left| \int (\mathcal{\Bar{L}}^*)^2 e^{s\Bar{\mathcal{L}}^*}\Pi_{i=1}^p \Tilde{f}_{t_k} (z_i) dz_2\cdots dz_p \right| \right]ds.
    \end{equation*}
    \end{small}
    As shown in Lemma \ref{lem: C2}, $\Tilde{f}_{t_k} \in C^2(z),$ hence it holds 
    $$\int |(\Tilde{\mathcal{L}}_k^*)^2 \Tilde{f}_{t_k}| dz\le C ,\quad \int \left| (\mathcal{\Bar{L}}^*)^2 \Pi _{i=1}^p \Tilde{f}_{t_k} (z_i)\right| dz_2 \cdots dz_p\le C,$$ 
    where $C=C(\kappa,\epsilon_\psi ,R,d).$ By Lemma \ref{lem: 4.2}, taking $g_1 = (\Tilde{\mathcal{L}}_k^*)^2  \Tilde{f}_{t_k}$ and $g_2 = (\mathcal{\Bar{L}}^*)^2 \Pi _{i=1}^p \Tilde{f}_{t_k} (z_i),$ one has
    \begin{equation*}
        \int |(\mathcal{\Tilde{L}}_k^*)^2 e^{\tau \mathcal{\Tilde{L}}_k^*}\Tilde{f}_{t_k}| dz \le  \int |(\mathcal{\Tilde{L}}_k^*)^2 \Tilde{f}_{t_k}| dz,
    \end{equation*}
    and
    \begin{equation*}
        \int \left| (\mathcal{\Bar{L}}^*)^2 e^{\tau \Bar{\mathcal{L}}^*} \Pi _{i=1}^p \Tilde{f}_{t_k} (z_i)\right| dz_1 \cdots dz_p \le \int \left| (\mathcal{\Bar{L}}^*)^2 \Pi _{i=1}^p \Tilde{f}_{t_k} (z_i)\right| dz_1 \cdots dz_p,
    \end{equation*}
    Using the exchangability of $\mathcal{L}^*$ and $e^{\tau \mathcal{L}^*},$ with $\mathcal{L}^*=\mathcal{\Tilde{L}}_k^*$ or $\mathcal{\Bar{L}}^*.$ 
    Then one has $ |\Tilde{S}(\Tilde{f}_{t_k}) - \mathcal{Q}_\infty (\Tilde{f}_{t_k})|_{TV}  \le C \tau^2.$ 
    
    For $\Tilde{S}(\Tilde{f}_{t_k}),$ by Lemma \ref{lem: suppS}, it holds
    \begin{equation*}
        \text{supp}[\Tilde{S}_\infty(\Tilde{f}_{t_k})]\subset
        B(0,C)\times B(0,C e^{-\kappa\psi_0 t}).
    \end{equation*}

    For the support of $\mathcal{Q}_\infty(\Tilde{f}_{t_k}),$ by Lemma \ref{lem: supp}, one knows that 
    \begin{equation*}
        \text{supp}[\mathcal{Q}_\infty(\Tilde{f}_{t_k})]\subset
        B(0,C)\times B(0,C e^{-\kappa\psi_0 t_k}).
    \end{equation*}
    
    Hence for any $z \in \text{supp}|\Tilde{S}(\Tilde{f}_{t_k}) - \Tilde{f}_{t_k}|,$ it holds $|z|\le C.$
    Then, by Lemma \ref{lem: TV} and applying similar analysis on Lemma \ref{lem: lmmG} Step 1, one has 
    $$ W_q (\mathcal{Q}_\infty (\Tilde{f}_{t_k}),\Tilde{S}(\Tilde{f}_{t_k})) \le C(\kappa,\epsilon_\psi , f_0,d) \tau^{\frac{2}{q}}.$$

\qed

\end{appendix}

\bibliographystyle{plain}
\bibliography{main}

\end{document}